\documentclass[article]{elsarticle}

\usepackage{lineno,hyperref}

\usepackage{amsmath,amsxtra,amssymb,latexsym,
epsfig,amscd,amsthm,fancybox,epsfig}
\usepackage[mathscr]{eucal}
\usepackage{graphicx}
\usepackage{epsfig} 
\usepackage{epstopdf}
\usepackage{cases}
\usepackage{subfig}
\usepackage{color}
\def\para#1{\vskip .4\baselineskip\noindent{\bf #1}}

\newtheorem{thm}{Theorem}[section]

\newtheorem{deff}{Definition}[section]

\newtheorem{lem}{Lemma}[section]
\newtheorem{prop}{Proposition}[section]
\theoremstyle{definition}

\newtheorem{hyp}[thm]{Hypotheses}
\theoremstyle{remark}
\newtheorem{rem}{Remark}

\numberwithin{equation}{section}

\newcommand{\eps}{\varepsilon}

\newcommand{\A}{\mathcal{A}}
\newcommand{\B}{\mathcal{B}}

\newcommand{\D}{\mathcal{D}}

\newcommand{\F}{\mathcal{F}}

\newcommand{\E}{\mathbb{E}}

\newcommand{\N}{\mathbb{N}}
\newcommand{\0}{\mathcal{O}}
\newcommand{\PP}{\mathbb{P}}

\newcommand{\R}{\mathbb{R}}

\newcommand{\abs}[1]{\left\vert#1\right\vert}

\newcommand{\norm}[1]{\left\Vert#1\right\Vert}

\numberwithin{equation}{section}
\newcommand{\lo}{\text{loc}}

\newcommand{\bed}{\begin{displaymath}}
\newcommand{\eed}{\end{displaymath}}
\newcommand{\bea}{\bed\begin{array}{rl}}
\newcommand{\eea}{\end{array}\eed}
\newcommand{\ad}{&\!\!\!\disp}

\newcommand{\barray}{\begin{array}{ll}}
\newcommand{\earray}{\end{array}}

\def\disp{\displaystyle}

\newcommand{\1}{\boldsymbol{1}}

\def\bar{\overline}
\def\hat{\widehat}
\def\a.s{\text{\;a.s.\;}}

\begin{document}
	
\begin{frontmatter}
	
	\title{Stochastic Lotka-Volterra Competitive Reaction-Diffusion Systems Perturbed by Space-Time White Noise: Modeling and
		Analysis\tnoteref{mytitlenote}}
	\tnotetext[mytitlenote]{This research was supported in part by the National Science Foundation under grant DMS-2114649.}
		
	\author[myaddress]{Nhu N. Nguyen}
	\ead{nguyen.nhu@uconn.edu}
	\author[myaddress]{George Yin\corref{mycorrespondingauthor}}
	\cortext[mycorrespondingauthor]{Corresponding author}
	\ead{gyin@uconn.edu}
	
	\address[myaddress]{Department of Mathematics, University of Connecticut, Storrs, CT
		06269, USA}

\begin{abstract}
Motivated by the traditional
Lotka-Volterra competitive models, this paper proposes and analyzes a class of stochastic reaction-diffusion partial differential equations. In contrast to the models in the literature, the new formulation enables spatial dependence of the species. In addition, the noise process is allowed to be space-time white noise.
In this work, well-posedness, regularity of solutions,
existence of density, and existence of an invariant measure for stochastic reaction-diffusion systems with non-Lipschitz and non-linear growth coefficients and multiplicative noise are considered. By combining the random field approach and infinite integration theory approach in SPDEs for mild solutions, analysis is carried out. Then this paper develops a Lotka-Volterra competitive system under general setting; longtime properties are studied with the help of newly developed tools in stochastic calculus.
	\end{abstract}
	
	
\end{frontmatter}

\section{Introduction}
Motivated by the classical Lotka-Volterra competitive model introduced  in 1925 by Lotka \cite{Lot25},
\begin{equation}\label{int-eq-1}
\begin{cases}
\dfrac{dU(t)}{d t}=U(t)\left(m_1-a_1 U(t)-b_1V(t)\right),\ t\geq 0,\\[1.5ex]
\dfrac{d V(t)}{d t}=V(t)\left(m_2-a_2 V(t)-b_2U(t)\right),\ t\geq 0,\\
U(0)=U_0, V(0)=V_0,
\end{cases}
\end{equation}
much effort has been devoted to studying
 and generalizing
 this type of equations in different directions.
In \eqref{int-eq-1},
$U(t),V(t)$ are the densities of competing species at time $t$;
$m_1,m_2$ are
the birth rates;
$a_1$, $a_2$
represent the rates of self-limitation,
 and
$b_1$, $b_2$ account for the rates of competition.
The motivation for the study comes from ecology and biology. For detailed biological and ecological background related to \eqref{int-eq-1} and its variants,
see \cite{Kes88,Smi68}.

If one takes into consideration of the
spatial inhomogeneity,
\eqref{int-eq-1} can be generalized to
\begin{equation}\label{int-eq-1-1}
\!\!\!\!\!\begin{cases}
\dfrac{\partial U(t,x)}{\partial t}=\Delta U(t,x) +U(t,x)(m_1(x)-a_1(x)U(t,x)-b_1(x)V(t,x)),
\ \R_+\times (0,1)
\\[2ex]
\dfrac{\partial V(t,x)}{\partial t}=\Delta V(t,x) + V(t,x)(m_2(x)-a_2(x)V(t,x)-b_2(x)U(t,x)),
\ \R_+\times(0,1)
\\[2ex]
\dfrac{\partial U}{\partial x}(t,0)=\dfrac{\partial U}{\partial x}(t,1)=\dfrac{\partial V}{\partial x}(t,0)=\dfrac{\partial V}{\partial x}(t,1)=0,\ t\geq 0,\\
U(0,x)=U_0(x), V(0,x)=V_0(x),\ 0< x< 1,
\end{cases}
\end{equation}
where
$U(t,x),V(t,x)$ represent the densities of species at time $t$ and location $x$,
{\color{blue}$m_i(x)$, $a_i(x)$, $b_i(x)$, for $i=1,2$ are functions defined on $[0,1]$, and
$\Delta$ is the Laplace operator.}
Such a
model is the so-called reaction-diffusion equation in PDEs community and has received increasing attention lately.
For instant,
 in \cite{CL84, GL94}, the authors considered the existence and uniqueness of the coexistence states; the works \cite{HLM02,LN12} aimed to understand completely the dynamics of the system;
the work  \cite{HLM05} studied small diffusion cases;
the works  \cite{HN13,HN13-2,HN16,HN16-2,HN17,LZZ19} treated variants of \eqref{int-eq-1-1}.

Along another direction,  noises are added to \eqref{int-eq-1} to capture the random factors in the environment. The corresponding stochastic system becomes
\begin{equation}\label{int-eq-1-2}
\begin{cases}
dU(t)=U(t)\left(m_1-a_1U(t)-b_1V(t)\right)dt+\sigma_1 U(t)dB_1(t),\quad t\geq 0,\\
dV(t)=V(t)\left(m_2-a_2V(t)-b_2U(t)\right)dt+\sigma_2 V(t)dB_2(t),\quad t\geq 0,\\
U(0)=U_0, V(0)=V_0,
\end{cases}
\end{equation}
where $B_1(t)$ and $B_2(t)$ are real-valued standard Brownian motions, and $\sigma_1,\sigma_2$ are intensities of the noises.
The system is modeled and studied under stochastic differential equations (SDEs) framework.
Much effort has been devoted to studying \eqref{int-eq-1-2} such as well-posedness, positivity of solution, Markov-Feller property,
longtime dynamic behavior such as
existence and uniqueness of stationary distribution, coexistence and extinction, and optimal harvesting strategy; see e.g., \cite{DD11,DDY14,HS98,DY17,TY15} and reference therein.

In this work,
we propose a model that
captures features of both the random factors and the spatial inhomogeneity.
Let
$\frac{\partial ^2W_1(t,x)}{\partial t\partial x}$, $\frac{\partial ^2W_2(t,x)}{\partial t\partial x}$ be space-time white noises, to be defined
rigorously
in the subsequent section;
$m_i(x)$, $a_i(x)$, $b_i(x)$, and $\sigma_i(x)$, for $i=1,2$ be twice continuously differentiable functions in $[0,1]$,
and suppose that $m_i(x)$, $a_i(x)$, and $b_i(x)$ are non-negative.
  Consider
\begin{equation}\label{eq-main}
\begin{cases}
\dfrac{\partial U(t,x)}{\partial t}=\Delta U(t,x) +U(t,x)\left(m_1(x)-a_1(x)U(t,x)-b_1(x)V(t,x)\right)\\[1ex]
\hspace{4cm}+ \sigma_1(x)U(t,x)\dfrac{\partial ^2W_1(t,x)}{\partial t\partial x},\quad 0\leq x\leq 1, t\geq 0,\\[1ex]
\dfrac{\partial V(t,x)}{\partial t}=\Delta V(t,x) + V(t,x)(m_2(x)-a_2(x)V(t,x)-b_2(x)U(t,x))\\[1ex]
\hspace{4cm}+\sigma_2(x)V(t,x)\dfrac{\partial^2 W_2(t,x)}{\partial t\partial x},\quad 0\leq x\leq 1, t\geq 0,\\[1ex]
\dfrac{\partial U}{\partial x}(t,0)=\dfrac{\partial U}{\partial x}(t,1)=\dfrac{\partial V}{\partial x}(t,0)=\dfrac{\partial V}{\partial x}(t,1)=0,\quad t\geq 0,\\
U(0,x)=U_0(x), V(0,x)=V_0(x), \quad 0\leq x\leq 1.
\end{cases}
\end{equation}
The use of Neumann boundary condition is motivated by applications in biology and ecology, namely, the population will not leave a finite domain.

Our results can be summarized as follows.
After modeling the system as a stochastic reaction-diffusion system perturbed by space-time white noise under a stochastic partially differential equation (SPDE) framework, we give a full analysis.
The well-posedness of the problem (existence, uniqueness, positivity, and continuous dependence on initial data of the solution)
is obtained first.
In contrast to many existing works,
we do not require
the coefficients being Lipschitz,
neither do we use linear growth condition. As a result, this part is also interesting in its own right from a SPDEs theory point of view.
Then, the regularity of the solution is investigated. It is shown that in any compact interval not including $0$, the solution satisfies the classical regularity, namely, H\"older continuous in the time variable with any exponent $<1/4$ and H\"older continuous in the space variable with any exponent $<1/2$, while on compact interval containing $0$, the H\"older continuity only holds with exponent $<1/2\wedge \alpha$ in space and with exponent $<1/4\wedge\alpha/2$ in time provided that the initial value is $\alpha$-H\"older continuous.
Analytic and probabilistic representations of heat kernel are used in the proof.
Next, using the Malliavin calculus, the absolute continuity with respect to Lebesgue measure of the law of the solution is proved and then the existence of density is obtained.
The longtime behavior is  also studied.
We prove the existence of an invariant measure.
Then, we consider an important problem in biology and ecology, namely, the coexistence and the extinction.
Some ideas and methods for the study
of this longtime property
are introduced and a first attempt is given by  using the newly developed mild stochastic calculus.
An overview of the results, ideas, and methods of this point in other (simpler) frameworks and the difficulties in our own system are also discussed carefully.
Finally,
we  extend our results to high dimensional setting by
injecting
``color'' (or correlation) into the space-time white noise for the trade off of the regularity of the noise and the dimension of space.
The noise driving the equation in higher dimensional space will be white in time and colored in space. Nevertheless, one  need not require the use of finite-trace covariance.

Regarding the novelty,
this paper is one of the first works on modeling and analysis of the competitive models in
biological system when both the spatial inhomogeneity and the random noises are taken into consideration.
Note that the systems in SPDEs setting cannot be investigated by simply
combining SDEs and PDEs.
For example, the stochastic integral with respect to space-time white noise requires
integrating over the time and space variables simultaneously. Roughly, if we frozen the time, it looks like a Bochner integral while if the space variable is frozen, it turns out to be an integral in the It\^o sense.
However, putting
them together will be different from considering and analyzing each of them separately. As a result, the analysis is much difficult compared with the existing results in either
SDEs or PDEs setting.

Our work contributes to both
 the development of stochastic reaction-diffusion equations and particular applications to Lotka-Volterra systems.
We consider well-posedness of the problem, regularity of the solutions, existence of density, existence of an invariant measure for a stochastic reaction-diffusion system with non-Lipschitz and non-linear growth coefficients and multiplicative noise.
Moreover, we use a unified approach by combining
the random
field
approach and the infinite integration theory approach in SPDEs for mild solutions.
Each of the approaches has its own advantage and is suitable for different purposes. From
 an application point of view, this paper models and analyzes the Lotka-Volterra competitive system in a  more general setting.
 The longtime properties are also studied with the help of newly developed tools in stochastic calculus.
We hope this work will open up a new window
 for studying biological systems as well as the applications of SPDEs in mathematical biology.

The rest of paper is organized as follows.
Section \ref{sec:for} provides the
formulation of our problem.
The well-posedness of the problem is given in Section \ref{sec:wel} while Section \ref{sec:reg} is devoted to the regularity of the solution.
The existence of the density of the law of the solution is obtained in Section \ref{sec:den}.
Section \ref{sec:inv} considers the existence of an invariant measure.
Section \ref{sec:lon} is devoted to the coexistence and extinction in the competitive model.
We extend our results to higher dimension in Section \ref{sec:hig}.
Section \ref{sec:con} concludes our paper. Finally, an appendix containing some notation
and  results together with relevant literature is provided at the end of the paper to help the reading, which includes
infinite-dimensional integration theory,
random field approach, equivalence of the two different approaches, and the Malliavin calculus.

\section{Formulation}\label{sec:for}
The driving noise  we consider has two parameters, space and time.
There are several ways to construct stochastic partial differential equations (SPDEs)
with respect to such noises.
The theory of SPDEs was developed  based on the random field approach by Walsh in \cite{Wal86}, and was dealt with
using
stochastic evolution in Hilbert space  by Da Prato and Zabczyk in \cite{PZ92}.
In the results developed by Walsh,  stochastic integrals are defined with respect to martingale measures, whereas in the work of Da Prato and Zabczyk, stochastic integrals are taken with respect to Hilbert space-valued Wiener processes.
These two approaches lead to the developments of two distinct schools of study for SPDEs, both of which have advantages in their own rights.

In this paper, we prove that the solutions in the two approaches
 for the systems that we are interested in are equivalent. Then we treat the solution in each sense exchangeably whichever is more convenient for us
 under different scenarios.
Unifying and using both approaches is one of our main ideas here and allows us to give a full analysis of the systems of interest.
For easy references on the aforementioned approaches, we collect some notation and preliminary results in the appendix.

To proceed,
let us formulate our problem.
Let $L^2((0,1),\R)$ be the Hilbert space with usual inner product and $C([0,1],\R)$ be the Banach space of continuous functions with the sup-norm.
 Denote by
 $H=L^2((0,1),\R^2)$ and $E=C([0,1],\R^2)$ the Hilbert space and Banach space, respectively, endowed with the inner product and the norm as follows
$$
\langle h,g\rangle_H=
\langle (h_1,h_2), (g_1,g_2)\rangle_{H}:=\sum_{i=1}^2 \langle h_i,g_i\rangle_{L^2((0,1),\R)},
$$
and
$$
|u|_E=|(u_1,u_2)|_{E}:=\sup_{x\in [0,1]}\sqrt{u_1^2(x)+u_2^2(x)}.
$$
Let $\big\{\Omega, \F,\{\F_t\}_{t\geq 0},\PP\big\}$ be a complete probability space and $L^p(\Omega;C([0,T],E))$ (resp. $L^p(\Omega;C([0,T],H))$) be the subspace of predicable process $u$, which take values in $C([0,T],E)$ (resp. $C([0,T],H)$) a.s. with the norm
$$\abs{u}^p_{L_{t,p}}:=\E \sup_{s\in [0,t]}\abs{u(s)}^p_{E},\quad (\text{resp. }\abs{u}^p_{L_{t,p}(H)}:=\E \sup_{s\in [0,t]}\abs{u(s)}^p_{H}).$$
For $\eps>0,p\geq 1$, denote by $W^{\eps,p}((0,1),\R^2)$ the Sobolev-Slobodeckij space (the Sobolev space with non-integer exponent) endowed with the norm
$$\abs{u}_{\eps,p}:=\abs{u}_{L^p((0,1),\R^2)}
+\sum_{i=1}^2\int_{(0,1)\times(0,1)}
\dfrac{\abs{u_i(x)-u_i(y)}^p}{\abs{x-y}^{\eps p+1}}dxdy.$$

\para{Neumann heat kernel and Neumann heat semi-group.}
Next, we denote by  $G_t(x,y)$ the fundamental solution of the heat equation on $\R_+\times (0,1)$ with the \textcolor{blue}{Neumann boundary condition}.
It is well known that $G_t(x,y)$ has an explicit form as follows
\begin{equation*}
\begin{aligned}
G_t(x,y)=\frac 1{\sqrt{4\pi t}}\sum_{n=-\infty}^{\infty}\Bigg[\exp\left(-\frac {(y-x-2n)^2}{4t}\right)+\exp\left(-\frac{(y+x-2n)^2}{4t}\right)\Bigg].
\end{aligned}
\end{equation*}
We  recall the following properties of the Neumann heat kernel; see e.g., \cite{BP98,Wal86}.
\begin{itemize}
\item There are some finite constants \textcolor{blue}{$c$} and \textcolor{blue}{$c'$} such that
\begin{equation}\label{prop-A1}
\textcolor{blue}{c} G_{t-s}(x,y)\leq \frac 1{\sqrt{2\pi (t-s)}}\exp\left(-\frac{|x-y|^2}{2(t-s)}\right)\leq \textcolor{blue}{c'}G_{t-s}(x,y),
\end{equation}
 where
 $$
\frac 1{\sqrt{2\pi (t-s)}}\exp\left(-\frac{|x-y|^2}{2(t-s)}\right)
$$
is the heat kernel.
\item
For each $q<3$, $T>0$, one has
\begin{equation}\label{prop-A2}
\sup_{(t,x)\in [0,T]\times [0,1]}\int_0^t \int_0^1 G_{t-s}^{q}(x,y)dyds<\infty.
\end{equation}
\end{itemize}
Moreover, let $e^{t\Delta_N}$ be a semigroup in $L^2((0,1),\R)$ defined by
$$
\left(e^{t\Delta_N}u\right)(x):=\int_0^1 G_t(x,y)u(y)dy.
$$
We recall some  properties of this semigroup as follows;  see \cite[Section 2.1]{Cer03} for more details.
\begin{itemize}
\item For any $t>0$, $\eps>0$, $p\geq 1$, $e^{t\Delta_N}$ maps $L^p((0,1),\R)$ into $W^{\eps,p}((0,1),\R)$ and
\begin{equation}\label{prop-S1}
\left|e^{t\Delta_N}u\right|_{\eps,p}\leq c(t\wedge 1)^{-\eps/2}|u|_{L^p((0,1),\R)},\quad \forall u\in L^p((0,1),\R),
\end{equation}
for some constant $c$ independent of $p$.
\item There is a constant $c$, independent of $u$ such that
\begin{equation}\label{prop-S2}
|e^{t\Delta_N}u|_{C([0,1],\R)}\leq c|u|_{C([0,1],\R)},\;\forall u\in C([0,1],\R).
\end{equation}
\end{itemize}
Moreover, we
often use the notation $e^{t\Delta_N}u$ for $u=(u_1,u_2)$,
in the following definition:
\begin{equation}\label{117-1}
e^{t\Delta_N}u:=\left(e^{t\Delta_N}u_1,e^{t\Delta_N}u_2\right).
\end{equation}
For simplicity of notation,
in the remaining of the paper, $e^{t\Delta_N}u$ with $u$ being a function taking $\R^2$ values, should be understood as in \eqref{117-1}.

\para{Space-time white driving noise.} Assume that $\{\beta_{1,k}(t)\}_{k=1}^\infty$, and $\{\beta_{2,i}(t)\}_{k=1}^\infty$ are two sequences of independent $\{\F_t\}_{t\geq 0}$-adapted one-dimensional standard Wiener processes. Now, let $\{e_k\}_{k=1}^{\infty}$ be a complete orthonormal system in $L^2((0,1),\R)$ including eigenfunctions of Neumann Laplace operator in $[0,1]$. It is seen that they are uniformly bounded. That is,
$$\sup_{k\in\N}\sup_{x\in[0,1]}\abs{e_k(x)}<\infty.$$
We define the standard cylindrical $Q$-Winner processes $W_i(t), i=1,2$  as follows
$$
W_i(t)=\sum_{k=1}^{\infty}\beta_{k,i}(t)e_k,\quad i=1,2.$$
In higher dimension, we will need to
use colored noise in space
to obtain more regularity
 but
do not need to require it be a finite-trace $Q$-Wiener process.
The detail
is discussed in Section \ref{sec:hig}.

\para{Definition of  solution.}
Now, we define a mild solution of \eqref{eq-main} as a process
$$\{Z(t,x):=(U(t,x), V(t,x)): t\geq 0,x \in (0,1)\}$$
 satisfying
\begin{equation}\label{eq-mild-random}
\begin{cases}
\displaystyle U(t,x)=\int_0^ 1G_t(x,y)U_0(y)dy\\
\hspace{1cm}+\displaystyle\int_0^t\int_0^1 G_{t-s}(x,y)U(s,y)\left(m_1(y)-a_1(y)U(s,y)-b_1(y)V(s,y)\right)dyds\\[1.5ex]
\hspace{1cm}+\displaystyle\int_0^t\int_0^1 G_{t-s}(x,y)\sigma_1(y)U(s,y)W_1(ds,dy),\\[1.5ex]
\displaystyle V(t,x)=\int_0^1G_t(x,y)V_0(y)dy\\
\hspace{1cm}+\displaystyle\int_0^t\int_0^1 G_{t-s}(x,y)V(s,y)\left(m_2(y)-a_2(y)V(s,y)-b_2(y)U(s,y)\right)dyds\\[1.5ex]
\hspace{1cm}+\displaystyle\int_0^t\int_0^1 G_{t-s}(x,y)\sigma_2(y)V(s,y)W_2(ds,dy),
\end{cases}
\end{equation}
where the stochastic integrals are in Walsh's sense with respect to the corresponding Brownian sheets of $W_1(t)$, $W_2(t)$ (denoted by $W_1(t,y), W_2(t,y)$ for simplicity of notation) as in Section \ref{sec-22} and \ref{sec-23};
or satisfying the following stochastic integral equation
\begin{equation}\label{eq-mild-infinite}
\begin{cases}
\displaystyle U(t)=e^{t\Delta_N}U_0+\int_0^t e^{(t-s)\Delta_N}U(s)\left(m_1-a_1U(s)-b_1V(s)\right)ds\\
\hspace{1.5cm}\displaystyle+\int_0^te^{(t-s)\Delta_N}\sigma_1U(s)dW_1(s),\\[2ex]
\displaystyle V(t)=e^{t\Delta_N}V_0+\int_0^t e^{(t-s)\Delta_N}V(s)
\left(m_2-a_2V(s)-b_2U(s)\right)
ds\\
\hspace{1.5cm}\displaystyle+\int_0^te^{(t-s)\Delta_N}\sigma_2V(s)dW_2(s),
\end{cases}
\end{equation}
where
the stochastic integrals, in which $\sigma_1U(s)$ and $\sigma_2V(s)$
as multiplication operators, are defined as in infinite-dimensional integration theory in Section \ref{sec-21}
and
$U(t)=U(t,x)$, $V(t)=V(t,x)$, $m_i=m_i(x)$, $a_i=a_i(x)$, $b_i=b_i(x)$, $\sigma_i=\sigma_i(x)$ ($i=1,2$) are understood as elements in a Hilbert space $L^2((0,1),\R)$.

As we discussed in Section \ref{sec-23}, these solutions (in the sense of \eqref{eq-mild-random} and of \eqref{eq-mild-infinite}) are equivalent if one of them exists uniquely and has continuous version and finite moment (it will be shown in Section \ref{sec:wel}).
Because of this equivalence, we will use these forms exchangeably depending on our purposes.
To prove the existence and uniqueness of the solutions, to examine their longtime behavior, or to obtain estimates in functional spaces, the solution in the sense of infinite-dimensional theory \eqref{eq-mild-infinite} will be used. To investigate the
the regularity of solution and its distribution or to estimate pointwise, we use the solution in the sense of random field approach \eqref{eq-mild-random}.

In the rest of paper, we often denote the functionals $F_1(U,V)$, $F_2(U,V)$ as the drift terms of \eqref{eq-main}.
For $u(x), v(x) \in L^2((0,1),\R)$, we say $u\geq 0$ if $u(x)\geq 0$ almost everywhere $x\in(0,1)$; and $u\geq v$ if $u(x)\geq v(x)$ almost everywhere $x\in (0,1)$.
Since we are treating a system motivated from ecological system and mathematical biology, we are only interested in ``non-negative mild solution", i.e., the mild solution $(U(t),V(t))$ satisfying $U(t)\ge 0$ and $V(t)\ge 0$ for all $t\geq 0$ \a.s
Moreover, operations with respect to vectors are understood in the usual sense although we will often write them in row instead of in column because of the simplicity of notations.
Throughout this paper, the letter $c$  denotes a generic finite positive constant whose values may change in
different occurrences. We will write the dependence of the constants on parameters explicitly when it is needed.

\section{Well-posedness}\label{sec:wel}
\para{Regularity of stochastic integral.}
To start,
it is similar to \cite{Cer03},
we need the following proposition, which shows the regularity of the stochastic integral.

\begin{prop}\label{s3-prop-1}
Denote by $\gamma$ the mapping
\begin{equation}\label{eq-gamma}
\gamma(u)(t):=\left(\int_0^t e^{(t-s)\Delta_N}\sigma_1u_1(s)dW_1(s);\int_0^t e^{(t-s)\Delta_N}\sigma_2u_2(s)dW_2(s)\right),
\end{equation}
for $u=(u_1,u_2)\in L^p(\Omega;C([0,T],E))$.
There is $p_*$ such that for all $p\geq p_*$, $\gamma$ maps $ L^p(\Omega;C([0,T],E))$ into itself and for any $u,v\in  L^p(\Omega;C([0,T],E))$
\begin{equation}\label{s3-eq-gamma}
|\gamma(u)-\gamma(v)|_{L_{T},p}\leq c_p(T)|u-v|_{L_{T,p}},
\end{equation}
for some function $c_p(T)$ satisfying $c_p(T)\downarrow 0$ as $T\downarrow 0$.
\end{prop}

\begin{proof}
Let $p_*$ be sufficiently large such that for any $p\geq p_*$,
we can choose simultaneously $\alpha,\eps>0$ satisfying
\begin{equation}\label{cond-epsbeta}
\frac 1p<\alpha<\frac 14\quad\text{and}\quad \frac 1p<\eps<2\big(\alpha-\frac 1p\big).
\end{equation}
By a factorization argument (see e.g., \cite[Theorem 8.3]{PZ92}), one has
\begin{equation}\label{s3-prop1-eq0}
\gamma(u)(t)-\gamma(v)(t)=\frac{\sin \pi\alpha}{\pi}\int_0^t(t-s)^{\alpha-1}
e^{(t-s)\Delta_N}Y_\alpha(u,v)(s)ds,
\end{equation}
where
$$
Y_\alpha(u,v)(s):=\int_0^s (s-r)^{-\alpha}e^{(s-r)\Delta_N}(u(r)-v(r))\sigma dW(r),
$$
and we shortened the notation by convention that for $u=(u_1,u_2)$
$$
\begin{aligned}
\int_0^t&e^{(t-s)\Delta_N} u(s)\sigma dW(s)\\
&:=\Big(\int_0^t e^{(t-s)\Delta_N}u_1(s)\sigma_1dW_1(s),\int_0^t e^{(t-s)\Delta_N}u_2(s)\sigma_2dW_2(s)\Big).
\end{aligned}
$$
Applying \eqref{prop-S1} and H\"oder's inequality to \eqref{s3-prop1-eq0} yields that for any $\eps<2(\alpha-1/p)$
\begin{equation}\label{s3-prop1-eq3}
\begin{aligned}
|\gamma(u)(t)&-\gamma(v)(t)|_{\eps,p}\\
&\leq c_\alpha\int_0^t ((t-s)\wedge 1)^{\alpha-1-\eps/2}|Y_\alpha(u,v)(s)|_{L^p((0,1),\R^2)}ds\\
&\leq c_\alpha\Big(\int_0^t (s\wedge 1)^{\frac{p}{p-1}(\alpha-1-\eps/2)}ds\Big)^{\frac{p-1}p}\Big(\int_0^t\left|Y_\alpha(u,v)(s)\right|^p_{L^p((0,1),\R^2)}ds\Big)^{\frac 1p}.
\end{aligned}
\end{equation}

We proceed to estimate $Y_{\alpha}(u,v)(s)$. First, let
$$
Y_{\alpha}^1(u,v)(s):=\int_0^s(s-r)^{-\alpha} e^{(s-r)\Delta_N}(u_1(r)-v_1(r))\sigma_1 dW_1(r).
$$
It is noted that
the stochastic convolution $\int_0^s e^{(s-r)\Delta_N}\Phi(r)dW(r)$ (for some process $\Phi$ such that the integral is well defined) is not a martingale with respect to $s$ in general.
However, if we frozen $s$ and consider the sequence $\int_0^{s'} e^{(s-r)\Delta_N}\Phi(r)dW(r)$ with respect to $s'\in [0,s]$, then it is a martingale.
Taking this idea, by Burkholder-Davis-Gundy inequality, we get
\begin{equation}\label{s3-prop1-eq1}
\begin{aligned}
\E&|Y_{\alpha}^1(u,v)(s,x)|^p\\
&\leq c\E\bigg(\int_0^s(s-r)^{-2\alpha}\sum_{k=1}^{\infty} \Big(\int_0^1 G_{s-r}(x,y)(u_1(r,y)-v_1(r,y))e_k(y)dy\Big)^2dr\bigg)^{\frac p2}\\
&= c\E\bigg(\int_0^s(s-r)^{-2\alpha}\sum_{k=1}^\infty \langle G_{s-r}(x,\cdot)(u_1(r,\cdot)-v_1(r,\cdot)), e_k(\cdot)\rangle_{L^2((0,1))}^2dr\bigg)^{\frac p2}\\
&=c\E\bigg(\int_0^s (s-r)^{-2\alpha}\left|G_{s-r}(x,\cdot)(u_1(r,\cdot)-v_1(r,\cdot))
\right|_{L^2((0,1))}^2dr\bigg)^{\frac p2}
\end{aligned}
\end{equation}
because of  Parseval's identity.
Moreover,
\begin{equation}\label{s3-prop1-eq2}
\begin{aligned}
\big|G_{s-r}(x,\cdot)&(u_1(r,\cdot)-u_2(r,\cdot))\big|_{L^2((0,1))}^2\\
&\leq |u_1(r)-v_1(r)|^2_{C([0,1],\R)}\int_0^1 G_{s-r}^2(x,y)dy\\
&\leq c|u_1(r)-v_1(r)|^2_{C([0,1],\R)}(s-r)^{-\frac 12}\quad(\text{due to \eqref{prop-A1}}).
\end{aligned}
\end{equation}
The second component $\E|Y_{\alpha}^2(u,v)(s,x)|^p$ is estimated similarly. Hence,
combining \eqref{s3-prop1-eq1} and \eqref{s3-prop1-eq2}
allows us to obtain that
\begin{equation}\label{s3-prop1-eq4}
\E|Y_{\alpha}(u,v)(s,x)|^p\leq c |u-v|^p_{L_{s,p}}\left(\int_0^s (s-r)^{-(2\alpha+\frac 12)}dr\right)^{\frac p2}.
\end{equation}
By the Sobolev embedding theorem, $W^{\eps,p}((0,1))$ is embedded into $C([0,1])$ if $\eps>1/p$. Hence, we deduce from \eqref{s3-prop1-eq3} and \eqref{s3-prop1-eq4} that  $\gamma$ maps $ L^p(\Omega;C([0,T],E))$ into itself,
 and
$$
|\gamma(u)-\gamma(v)|_{L_{t,p}}\leq c_p(t)|u-v|_{L_{t,p}},
$$
where
$$
c_p(t)=c_\alpha\Big(\int_0^t (s\wedge 1)^{\frac{p}{p-1}(\alpha-1-\eps/2)}ds\Big)^{\frac{p-1}p}\Big(\int_0^t\Big(\int_0^s (s-r)^{-(2\alpha+\frac 12)}dr\Big)^{\frac p2}ds\Big)^{\frac 1p}
$$
satisfying $c_p(t)\downarrow 0$ as $t\downarrow 0$ due to \eqref{cond-epsbeta}.
\end{proof}

\para{Existence and uniqueness of solutions.}
Since the coefficients are non-Lipschitz and non-linear growth, the existence and uniqueness of the mild solution are not obvious as usual. Although the existence of the mild solution of stochastic reaction-diffusion equations with non-Lipsschitz terms has been obtained in \cite{Cer03}, we cannot apply the results in this paper because
the coefficients in \cite{Cer03} are required to satisfy either some suitable growth conditions \cite[Hypothesis 4 and Theorem 5.3]{Cer03} or condition \cite[(5.17)]{Cer03}. That is, either the drift term has growth rate of power $m$ (for some $m>0$)
 resulting in the diffusion term having growth rate of at most power $\frac 1 m$ or the drift decays outside large balls, which
are needed to guarantee the (uniform) boundedness of the sequence of truncated solutions.
These conditions are not satisfied in our model since
the drift has polynomial growth of degree $2$,  while the diffusion term has linear growth. Moreover, the drift term in our own system only satisfies \cite[(5.17)]{Cer03} if we assume further conditions
 such as $a_1(x),a_2(x)$ are uniformly bounded below by positive numbers;
see Section \ref{sec:inv}.

Given the problem mentioned above,
 we proceed as follows. We use the truncation method as in \cite{Cer03} to truncate the coefficients in compact balls so that it is Lipschitz continuous and linear growth, and
we define the solution using the truncation.
The non-negativity of truncated solutions will be obtained
in the next step. Then, the uniform boundedness of sequence of truncated solutions is obtained by using the idea of ``ignoring negative terms in the drift".
The detail is in the next Theorem.

\begin{thm}\label{thm-exi}
For any initial data $0\leq U_0,V_0$ with $(U_0,V_0)\in E$, there exists a unique mild solution $Z(t)=(U(t),V(t))$ of \eqref{eq-mild-infinite} in $L^p(\Omega; C([0,T],E))$ for any $T>0$, $p\geq 1.$
Moreover, $U(t),V(t)\geq 0, \forall t\geq 0\a.s$
\end{thm}

\begin{proof}
First, we rewrite the coefficients by defining
\bea \ad f_1(x,u,v)=u\big(m_1(x)-a_1(x)u-b_1(x)v\big),\\
\ad f_2(x,u,v)=v\big(m_2(x)-a_2(x)v-b_2(x)u\big),\eea
where $f_i: [0,1] \times \R\times \R \rightarrow \R$.
For each $n\in\N$, $i=1,2$, we define
\[f_{n,i}(x,u,v):=
\begin{cases}
f_i(x,u,v)\quad \text{if}\quad\abs{(u,v)}_{\R^2}\leq n,\\
f_{i}\Big(x,\dfrac {nu}{\abs{(u,v)}_{\R^2}},\dfrac {nv}{\abs{(u,v)}_{\R^2}}\Big) \quad \text{if}\quad \abs{(u,v)}_{\R^2}> n.\\
\end{cases}
\]
For each $n$,
$f_n(x,\cdot,\cdot)=\big(f_{n,1}(x,\cdot,\cdot),f_{n,2}(x,\cdot,\cdot)\big):\R^2\rightarrow \R^2$ is Lipschitz continuous,
uniformly with respect to $x\in [0,1]$, so that the composition operator $F_n(z)$ associated to $f_n$ (with $z(x)=(u(x),v(x))$) defined by
$$F_n(z)(x)=:\big(F_{n,1}(z)(x),F_{n,2}(z)(x)\big):=
\big(f_{n,1}(x,z(x)),f_{n,2}(x,z(x))\big), x\in[0,1],$$
is Lipschitz continuous in both $L^2((0,1), \R^2)$ and $C([0,1],\R^2)$.

We proceed to consider the following problem
\begin{equation}\label{Z_n}
dZ_n(t)=\big[\Delta_N Z_n(t)+F_n(Z_n(t))\big]dt+\sigma Z_n(t)dW(t), \quad Z_n(0)=(U_0,V_0),
\end{equation}
where $Z_n(t)=\big(U_n(t),V_n(t)\big)$, $\Delta_N Z_n(t):=\big(\Delta_N U_n(t),\Delta_N V_n(t)\big)$, and $\Delta_N$ is the Laplacian
together with the Neumann boundary condition and
$$\sigma Z_n(t)dW(t):=
\big(\sigma_1U_n(t)dW_1(t),\sigma_2V_n(t)dW_2(t)\big).$$

\begin{lem}\label{exist-lipchitz}
For any initial condition $(U_0,V_0)\in E$,
\eqref{Z_n} has a unique mild solution;
the solution is in $ L^p(\Omega;C([0,T],E))$ for any $p\geq p_*$ and $T>0$.
\end{lem}

\begin{proof}
Since the coefficients in \eqref{Z_n} are Lipschitz continuous and because of Proposition \ref{s3-prop-1}, by contraction mapping argument  \cite[Proof of Theorem 3.1]{NY20} or \cite{PZ92}, we obtain that equation \eqref{Z_n} admits a unique mild solution $Z_n(t)=(U_n(t),V_n(t))\in L^p(\Omega; C([0,T_0],E))$ for some sufficiently small $T_0$. Therefore, for any finite $T>0$, there is a unique mild solution of \eqref{Z_n} in $L^p(\Omega; C([0,T],E))$ by repeating the arguments in $[T_0,2T_0]$, $[2T_0,3T_0]$, and so on.
\end{proof}

Next, we
prove the non-negativity of $U_n(t),V_n(t).$

\begin{lem}\label{positivity}
For any  initial condition $0\leq U_0,V_0$, $(U_0,V_0)\in E$, one has $U_n(t),V_n(t)\geq 0,\;\forall t\in[0,T] \a.s$
\end{lem}

\begin{proof}
The proof is similar to the proof of \cite[Lemma 3.1]{NY20} or \cite[Lemma 3.2]{NNY18}.
\end{proof}

 We are in a position to show that the sequence $\{Z_n\}_{n=1}^{\infty}$ is uniformly bounded. The result is in
 the following lemma.

\begin{lem}\label{lem-bounded}
For all $n\in\N$,
\begin{equation}\label{bounded}
\E \sup_{s\in [0,t]}\abs{Z_n(s)}_{E}^p\leq c_p(t)\big(1+\abs{Z_0}_{E}^p\big),
\end{equation}
where $c_{p}(t)$ is a positive constant depending on $p$ and $t$, but independent of $n$.
\end{lem}

\begin{proof}
 By the definition of mild solution, we have
\begin{equation*}
\begin{aligned}
U_n(t)(x)&=\left(e^{t\Delta_N}U_0\right)(x)+\left(\int_0^t e^{(t-s)\Delta_N}F_{n,1}(U_n(s),V_n(s))ds\right)(x)+W_{U_n}(t)(x),
\end{aligned}
\end{equation*}
where $W_{U_n}(t):=\int_0^t e^{(t-s)\Delta_N}\sigma_1U_n(s)dW_1(s).$
Since $e^{t\Delta_N}$  is positivity preserving and $U_n(t), V_n(t)$ are non-negative, by definition of $F_{n,1}$ and \eqref{prop-S2}, we obtain
\begin{equation}\label{U_n^3}
\begin{aligned}
&\abs{U_n(t)}_{C([0,1],\R)}
\\&=\sup_{x\in[0,1]}\Big[\left(e^{t\Delta_N}U_0\right)(x)+\left(\int_0^te^{(t-s)\Delta_N}F_{n,1}(U_n(s),V_n(s))ds\right)(x)+W_{U_n}(t)(x)\Big]
\\&\leq \sup_{x\in[0,1]}\Big[\left(e^{t\Delta_N}U_0\right)(x)+\left(\int_0^{t}e^{(t-s)\Delta_N}U_n(s)m_1ds\right)(x)+W_{U_n}(t)(x)\Big]
\\&\leq c(t)\Big(\big|U_0\big|_{C([0,1],\R)}+\int_0^t\Big|U_n(s)\Big|_{C([0,1],\R)}ds
+\Big|W_{U_n}(t)\Big|_{C([0,1],\R)}\Big),
\end{aligned}
\end{equation}
where $c(t)$ is a constant depending only on $t$ and independent of $n$.

There is a small $t_0>0$ such that
$$
c^\gamma_{p}(t_0)c_0(t_0)<\frac 12,
$$
where $c^\gamma_{p}(t_0)$ is the constant in \eqref{s3-eq-gamma} in Proposition \ref{s3-prop-1} and $c_0(t_0)$ is the constant in the last line of \eqref{U_n^3}.
Hence, we obtain from \eqref{U_n^3} and Proposition \ref{s3-prop-1} that
\begin{equation}
\E\sup_{s\in[0,t_0]}\abs{Z_n(s)}^p_{E}\leq c_p(t_0)\Big(\abs{Z_0}_{E}^p+\int_0^{t_0}\E\sup_{r\in [0,s]}\Big|Z_n(r)\Big|_{E}^pds\Big).
\end{equation}
Therefore, Gronwall's inequality implies that
$$\E\sup_{s\in[0,t_0]}\abs{Z_n(s)}^p_{E}\leq c_{p}(t_0)\big(1+\abs{Z_0}^p_{E}\big),$$
for some constant $c_{p}(t_0)$, independent of $n$.
To proceed, we can repeat the same arguments in the intervals $[t_0, 2t_0]$, $[2t_0,3t_0]$, and so on.
Thus the Lemma is proved.
\end{proof}

\para{Completion of the proof of Theorem \ref{thm-exi}.}
At this stage, we are able to define the solution using the truncation  \cite{Cer03} as follows.
For any $n\in\N$, we define
\begin{equation}\label{tau}
\zeta_n:=\inf\{t\geq 0: \abs{Z_n(t)}_{E} \geq n\},
\end{equation}
with the usual convention that $\inf \emptyset =\infty$, and define $\zeta=\sup_{n\in\N}\zeta_n.$ Then we have
$$\PP\{\zeta<\infty\}=\lim_{T\to\infty}\PP\{\zeta<T\},$$
and for each $T\geq 0$,
$$\PP\{\zeta\leq T\}=\lim_{n\to\infty}\PP\{\zeta_n\leq T\}.$$
For any fixed $n\in\N$ and $T\geq 0,$ it follows from Lemma \ref{lem-bounded} that
$$
\begin{aligned}
\PP\{\zeta_n\leq T\}&=\PP\Big\{\sup_{t\in [0,T]}\abs{Z_n(t)}^p_{E}\geq n^p\Big\}\\
&\leq \dfrac 1{n^p} \E \sup_{t\in [0,T]}\abs{Z_n(t)}^p_{E}\leq \dfrac{c_{p}(T)\big(1+\abs{Z_0}_{E}^p\big)}{n^p}.
\end{aligned}
$$
It leads to that $\PP\{\zeta_n\leq T\}$ goes to zero as $n\to\infty$ so $\PP\{\zeta=\infty\}=1.$ Hence, for any $t\geq 0$, and $\omega\in \{\zeta=\infty\}$, there exists an $n=n(\omega)\in \N$ such that $t\leq \zeta_n(\omega)$.
Thus we can define
\begin{equation}\label{wel-eq-def}
Z(t)(\omega):=Z_n(t)(\omega).
\end{equation}
We  need to show that
it
is well defined, i.e., for any $t\leq \zeta_n\wedge\zeta_m$,  $Z_n(t)=Z_m(t)$ a.s.
 This is because of the definitions of truncated coefficients and stopping times $\zeta_n,\zeta_m$.
The details of this argument can be found in
\cite[Theorem 5.3]{Cer03}.

Note that the process $Z(t)=(U(t),V(t))$ defined  above is a mild solution of \eqref{eq-main}.
Indeed, for any $t\geq 0$, $\omega \in \{\zeta=\infty\}$, there exists an $n\in\N$ such that $t\leq\zeta_n$ and
\begin{equation*}
\begin{aligned}
Z(t)&=Z_n(t)=e^{tA}Z_0+\int_0^t e^{(t-s)A}F_n(Z_n(s))ds+W_{Z_n}(t)
\\&=e^{tA}Z_0+\int_0^t e^{(t-s)A}F(Z(s))ds+W_Z(t),
\end{aligned}
\end{equation*}
{\color{blue}where $W_{Z_n}(t):=(W_{U_n(t)},W_{V_n}(t))$ and $W_Z(t)=(W_U(t),W_V(t))$.}
Moreover, if there exists another solution $\hat Z(t)$ of \eqref{eq-main},
it is not difficult to obtain that
$$Z(t\wedge \zeta_n)=\hat Z(t\wedge \zeta_n),\quad\forall n\in\N,t\geq 0.$$
Since $\zeta_n\to \infty$ as $n\to\infty$ a.s., we get $Z(t)=\hat Z(t).$
So, the solution is unique.
Finally, for any $p\geq 1, T>0$,
$$
\begin{aligned}
\sup_{t\in[0,T]}\abs{Z(t)}_{E}^p=\lim_{n\to\infty}\sup_{t\in[0,T]}\abs{Z(t)}_{E}^p \1_{\{T\leq\zeta_n\}}
=\lim_{n\to\infty}\sup_{t\in[0,T]}\abs{Z_n(t)}_{E}^p \1_{\{T\leq\zeta_n\}}.
\end{aligned}
$$
Hence, by the boundedness of $Z_n(t)$ in Lemma \ref{lem-bounded}, one has $Z(t)\in L^p(\Omega; C([0,T],E))$.
As a result,
we obtain that equation \eqref{eq-main} admits a unique mild solution $Z(t)=(U(t),V(t))\in L^p(\Omega; C([0,T],E))$.
The non-negativity of $U(t),V(t)$ follows
from that of $U_n(t),V_n(t)$.
\end{proof}

\para{Continuous dependence on initial data.}
To proceed, we prove that the solution depends continuously on initial data,
which is stated in the following Proposition.
This property  plays an important role in studying  the semigroup associated with the solution and its Feller property, which will be investigated in Section \ref{sec:inv}.

\begin{prop}\label{wel-prop-2}
The solution given in Theorem \ref{thm-exi}
depends continuously on initial data in the sense that
for any $T>0$, $p\geq 1$ the map
$
z\in E_+:=\{z=(u,v)\in E: u,v\geq 0\}\mapsto Z^z\in L^p(\Omega; C([0,T],E)),
$
$($where $Z^z(t)$ is the solution of \eqref{eq-main} with initial data $z$$)$
is continuous, uniformly on bounded sets of $E_+$.
\end{prop}

\begin{proof}
	With the help of \eqref{bounded}, the proof is similar to \cite[Proposition 5.6]{Cer03}. Thus we provide a sketch of the main ideas only.
Let $Z^{z_1}(t),Z^{z_2}(t)$ and $Z_n^{z_1}(t),Z_n^{z_2}(t)$ be the solutions of \eqref{eq-main} and \eqref{Z_n} with initial data $Z(0)=Z_n(0)=z_1$ and $Z(0)=Z_n(0)=z_2$, respectively. As in the  proof of the first part, because of the Lipschitz continuity of $F_n$, it is easy to obtain that
\begin{equation}\label{depend}
\abs{Z_n^{z_1}-Z_n^{z_2}}_{L_{T,p}}^p\leq c_{n,p}(T) \abs{z_1-z_2}_{E}^p.
\end{equation}
Consider the stopping times $\zeta_n^{z_1}$ and $\zeta_n^{z_2}$ as in \eqref{tau} corresponding to initial values $z_1, z_2$, respectively, we have
\begin{equation}\label{Zu-Zv}
\begin{aligned}
\abs{Z^{z_1}-Z^{z_2}}_{L_{T,p}}^p
\leq& \abs{Z_n^{z_1}-Z_n^{z_2}}_{L_{T,p}}^p \\
&+c_p\big(1+\abs{Z^{z_1}}_{L_{T,2p}}^p+\abs{Z^{z_2}}_{L_{T,2p}}^p\big)\big(\PP\{\zeta_n^{z_1}\wedge\zeta_n^{z_2}\leq T\}\big)^{1/2}.
\end{aligned}
\end{equation}
Moreover, it follows from \eqref{bounded} that
\begin{equation*}
\begin{aligned}
\PP\{\zeta_n^{z_1}\wedge\zeta_n^{z_2}\leq T\}
\leq \dfrac {c(T)}{n^2}\Big(1+\abs{z_1}^2_{E}+\abs{z_2}^2_{E}\Big).
\end{aligned}
\end{equation*}
Therefore, by applying  \eqref{bounded} again, we obtain from \eqref{Zu-Zv} and \eqref{depend} that
\begin{equation}\label{Z^u-Z^v}
\abs{Z^{z_1}-Z^{z_2}}_{L_{T,p}}^p\leq c_{n,p}(T)\abs{z_1-z_2}_{E}^p+\dfrac{c(T)}{n}\Big(1+\abs{z_1}_{E}^{p+1}+\abs{z_2}_{E}^{p+1}\Big).
\end{equation}
Now, for any $z_1,z_2$ in a bounded set of $E_+$ and arbitrary $\eps>0$, we first find $\bar n\in\N$ such that
$$\dfrac{c(T)}{\bar n}\Big(1+\abs{z_1}_{E}^{p+1}+\abs{z_2}_{E}^{p+1}\Big)<\dfrac{\eps}2,$$
where $c(T)$ is the constant in \eqref{Z^u-Z^v}.
Let $0<\delta<1$ be such that
$$c_{\bar n,p}(T)\abs{z_1-z_2}_{E}^p<\dfrac{\eps}2
\quad\text{whenever}\quad\abs{z_1-z_2}_{E}<\delta,$$
where $c_{\bar n,p}(T)$ is the constant in \eqref{Z^u-Z^v} corresponding to $\bar n$.
Therefore, continuous dependence of the solution on initial data is proved.
\end{proof}

\para{Positivity.}
We have obtained that the solutions are non-negative  provided the initial data are non-negative.
In fact, we expect that the solution to be positive under weak conditions on positivity of the initial data.
This property is also interesting in both SPDEs theory and different applications. Moreover, the results and techniques in this points will also be used to examine the existence of the density (of the law of solution) in Section \ref{sec:den}.
We have the following Proposition.

\begin{prop}\label{s3-prop-pos}
	Suppose $(U_0,V_0)\in E$ such that $U_0,V_0\geq 0$ but not identical to $0$. Then, $U(t),V(t)> 0, \forall t>0; \a.s$
\end{prop}

\begin{proof}
	Since $U(t),V(t)$ is  continuous a.s., it suffices to prove that for any fixed $t$, $U(t),V(t)>0$ a.s.
	We use truncation schemes as in Theorem \ref{thm-exi}.
	As in \cite[Theorem 2]{MN08} or \cite[Proposition 3.1]{PT93}, we
	 obtain the positivity of $ U_n(t), V_n(t)$, i.e., for any $t> 0$, $ U_n(t)$ and $V_n(t)>0$ a.s.
	Let $t>0$ be fixed but otherwise arbitrary. Since there are countable number of
	truncated equations, the set in which the positive property does not hold for some truncated solution is a null set.
	Therefore, because of the definitions of $U(t),V(t)$
 and $U_n(t),V_n(t)$,
 one can see that outside a null set, there is an $n=n(\omega)$ such that $U(t)=U_n(t),V(t)=V_n(t)$.
So, the positivity of $U(t), V(t)$ follows from the positivity of $U_n(t),V_n(t)$.	
\end{proof}

\para{Discussion on initial condition.}
To close this section, we discuss briefly conditions on the initial data for the existence and uniqueness of the mild solution.
In fact, the initial values are required to be in $E$, the space of continuous functions to guarantee the well-posedness of the problem in $E$.
If one only wants to obtain the well-posedness in the space of square integrable function $H$, the required continuity of initial condition is
not needed.
Let us state
this fact in the following theorem.

\begin{thm}\label{thm-exi-H}
For any initial data $0\leq U_0,V_0$, $(U_0,V_0)\in L^\infty((0,1),\R^2)$, there exists a unique mild solution $Z(t)=(U(t),V(t))$ of \eqref{eq-mild-infinite} belonging to $L^p(\Omega; C([0,T],H))$ for any $T>0, p\geq 1.$ The solution is non-negative, i.e., $U(t),V(t)\geq 0$
for any $t\geq 0$ \textcolor{blue}{\a.s} Moreover, the solution depends continuously on the initial data.
\end{thm}

\begin{proof}
The proof of this Theorem is the same as that of Theorem \ref{thm-exi}.
The truncation functions are defined first and then the sequence of truncated solutions are obtained.
We need only take care the uniform boundedness of the sequence of truncated solutions.
By the same arguments as that of Lemma \ref{lem-bounded}, we have the following Lemma.

\begin{lem}\label{wel-lem4}
For all $n\in\N$ then
\begin{equation*}
\E \sup_{s\in [0,t]}\abs{Z_n(s)}_{L^{\infty}((0,1),\R^2)}^p\leq c_p(t)\big(1+\abs{Z_0}_{L^{\infty}((0,1),\R^2)}^p\big),
\end{equation*}
where $c_{p}(t)$ is a positive constant depending on $p$ and $t$ but is independent of $n$.
\end{lem}

With this boundedness, we can mimic the remaining  proof of Theorem \ref{thm-exi}
to obtain the desired results.
\end{proof}

\begin{rem}{\rm
One may expect that to obtain the well-posedness in $H$, the initial condition $Z_0$ is required only to be in $H$.
However, this does not seem possible to us now.
The truncation process may be unavoidable in non-Lipschitz cases. Then, the uniform boundedness of sequence of truncated solutions in $L^\infty((0,1),\R^2)$ is needed to guarantee the solution to be
 well defined.
The uniform boundedness
in $H$
is
not enough
and thus,
the initial conditions need to be almost everywhere bounded.
}\end{rem}

\section{Regularity of solution}\label{sec:reg}
In Section \ref{sec:wel}, we have proved the existence and uniqueness of the solution belonging to the space of continuous functions.
In this section, we obtain additional regularities of the solution.

\para{H\"older continuity on $t>0$.}
We consider the H\"older continuity of the solution on intervals excluding $0$ first.
On these intervals, the solution satisfies the classical regularity, namely, H\"older continuity with exponent $<1/2$ in space and  exponent $<1/4$ in time.

\begin{thm}\label{reg-thm1}
Let $Z(t)$ be the solution of \eqref{eq-main} with initial value $Z_0=(U_0,V_0)\in E$, $U_0,V_0\geq 0$.
On compact set of $\{t>0\}$, the function $Z(t,x)$ is H\"older continuous in space with any exponent $<\frac 12$ and H\"older continuous in time with any exponent $<\frac 14$.
That is,
for any $0<t_0<T<\infty$, and $\beta_1\in (0,1/2)$, $\beta_2\in (0,1/4)$, there is a
finite random variable $C_H=C_H(t_0,T,\beta_1,\beta_2)$ a.s. such that
$$
|Z(t,x)-Z(s,y)|\leq C_H\left(|x-y|^{\beta_1}+|t-s|^{\beta_2}\right),\quad\forall x,y\in [0,1], s,t\in [t_0,T]\a.s
$$
\end{thm}

We need the following auxiliary results to prove Theorem \ref{reg-thm1}.

\begin{prop}\label{reg-prop1} $($Kolmogov's test; see e.g., {\rm\cite[Theorem 3.5]{PZ92}}$)$
Let $\0\subset\R^d$ be a bounded domain. There are $C$, $\delta$, and $\eps>0$ such that
$$
\E \left|X(\xi)-X(\eta)\right|^\delta\leq C|\xi-\eta|^{d+\eps}.
$$
Then $X(\cdot)$ has a H\"older continuous modification $($with any exponent $<\eps/\delta)$.
\end{prop}

\begin{lem}\label{reg-lem1}
We have the following basic property of Neumann heat kernel. For any $0<s<t<T$, $x,y\in [0,1]$,
one has
\begin{equation}\label{reg-eq-1-00}
\int_0^1\Big(G_{t}(x,\xi)-G_t(y,\xi)\Big)^2d\xi \leq \frac{C|x-y|^2}{t^\frac 32},
\end{equation}
\begin{equation}\label{reg-eq-1}
\int_0^t\left(\int_0^1\Big(G_{s}(x,\xi)-G_s(y,\xi)\Big)^2d\xi\right)ds\leq C|x-y|,
\end{equation}
\begin{equation}\label{reg-eq-2}
\int_s^t \left(\int_0^1 G^2_{t-r}(x,\xi)d\xi\right)dr\leq C |t-s|^{\frac 12},
\end{equation}
\begin{equation}\label{reg-eq-2-11}
\int_0^s\int_0^1\Big(G_{t-r}(x,\xi)-G_{s-r}(x,\xi)\Big)^2d\xi dr\leq C|t-s|^{\frac 12}.
\end{equation}
Moreover, for any $0<T_1<T_2$, there is a $C=C(T_1,T_2)$ such that
\begin{equation}\label{reg-eq-2-00}
\int_0^1 \left(G_t(x,\xi)-G_s(x,\xi)\right)^2d\xi\leq C|t-s|^{\frac 12},\;\forall s,t\in [T_1,T_2], x\in [0,1].
\end{equation}
\end{lem}

\begin{proof}
The proof is standard and can be found in \cite{Wal86}.
For example, one can obtain these results by using the eigenfunction expansions \cite{Wal86} of $G_t(x,y)$ in the form
$$
1+\sum_{n=1}^\infty 2e^{-n^2\pi^2t}\cos(n\pi x)\cos(n\pi y).
$$
Therefore, we have
\begin{equation}\label{reg-eq-1234}
\begin{aligned}
\int_0^1&\Big(G_{t}(x,\xi)-G_t(y,\xi)\Big)^2d\xi \\&=\sum_{n=1}^\infty 2e^{-2n^2\pi^2t}|\cos(n\pi x)-\cos(n\pi y)|^2,\quad(\text{due to Parseval's identity})\\
&\leq C\sum_{n=1}^\infty e^{-2n^2\pi^2t}\left(4\wedge n^2|x-y|^2\right)
\leq C\int_1^\infty e^{-t\xi^2}\left(4\wedge \xi^2|x-y|^2\right)d\xi\\
&\leq C|x-y|^2\int_1^\infty e^{-t\xi^2}\xi^2d\xi\leq \frac{C|x-y|^2}{t^{\frac 32}}.
\end{aligned}
\end{equation}
As a result, \eqref{reg-eq-1-00} is proved.
Moreover, it follows from \eqref{reg-eq-1234} that
$$
\begin{aligned}
\int_0^t \int_0^1\Big(G_{s}(x,\xi)-G_s(y,\xi)\Big)^2d\xi ds&\leq
C\int_0^t \int_1^\infty e^{-s\xi^2}\Big(4\wedge \xi^2|x-y|^2\Big)d\xi ds\\
&\leq C\int_1^\infty\Big(\frac 4{\xi^2}\wedge |x-y|^2\Big)d\xi
\leq C|x-y|.
\end{aligned}
$$
As a consequence, \eqref{reg-eq-1} is proved. Similarly, inequalities  \eqref{reg-eq-2} and \eqref{reg-eq-2-11} are obtained.
Finally, \eqref{reg-eq-2-00} can be proved by the same way as that of \eqref{reg-eq-1234} and using the fact $1-e^{-n^2(t-s)}\leq 1\wedge n^2(t-s)$.
\end{proof}

\begin{proof}[Proof of Theorem \ref{reg-thm1}]
It is known that
\begin{equation}\label{reg-eq1}
\begin{aligned}
Z(t,x)=&\left(e^{t\Delta_N}Z_0\right)(x)+\Big(\int_0^te^{(t-s)\Delta_N}F(Z(s))ds\Big)(x)+\gamma(Z)(t,x)\\
=:&A_0(t,x)+A_1(t,x)+\gamma(Z)(t,x),
\end{aligned}
\end{equation}
where $\gamma$ is the mapping defined as in \eqref{eq-gamma}.

First, by \eqref{reg-eq-1-00} and \eqref{reg-eq-2-00} in Lemma \ref{reg-lem1},
it can be seen that for any $k\in\N$,
\begin{equation}\label{reg-eq2}
\E|A_0(t,x)-A_0(t,y)|^{2k}\leq \frac{c_k|Z_0|_{E}^{2k}}{t^{\frac {3k}2}}|x-y|^k,\quad\forall t\geq 0, x,y\in[0,1],
\end{equation}
\begin{equation}\label{reg-eq2'}
\E|A_0(t,x)-A_0(s,x)|^{4k}\leq c_{k,t_0,T}|Z_0|_{E}^{4k}|t-s|^k,\quad\forall t,s\in [t_0,T], x,y\in [0,1].
\end{equation}
Combining \eqref{reg-eq2}, \eqref{reg-eq2'}, and Proposition \ref{reg-prop1}, we obtain the H\"older continuity of $A_0(t,x)$ in space with any exponent $<\frac 12$ and in time with any exponent $<\frac 14$.

Second, we have that $A_1(t,x)-A_1(t,y)$ consists of two components with one of them being
\begin{equation}\label{reg-eq3}
\begin{aligned}
\int_0^t \int_0^1 &G_{t-s}(x,\xi)F_1(Z(s,\xi))d\xi ds-\int_0^t \int_0^1 G_{t-s}(y,\xi)F_1(Z(s,\xi))d\xi ds\\
&=\int_0^t \left(\int_0^1 \left(G_{t-s}(x,\xi)-G_{t-s}(y,\xi)\right)F_1(Z(s,\xi))d\xi \right)ds,
\end{aligned}
\end{equation}
 Because $F_1(Z)$ has polynomial growth with the boundedness of solutions in the sense of
for any $p\geq 1$
$$\E \sup_{t\in [0,T]}|Z(t)|^p_E<\infty,$$
we have that for all $t\leq T$, $x,y\in [0,1]$, $k\in\N$,
\begin{equation}\label{reg-eq4}
\begin{aligned}
&\E\Big|\int_0^t \Big(\int_0^1 \left(G_{t-s}(x,\xi)-G_{t-s}(y,\xi)\right)F_1(Z(s,\xi))d\xi \Big)ds\Big|^{2k}\\
&\quad\leq c_k(T)\Big|\int_0^t \Big(\int_0^1 |G_{t-s}(x,\xi)-G_{t-s}(y,\xi)|^2d\xi\Big)ds\Big|^k\\
&\quad\leq c_k(T)|x-y|^k\quad\text{due to \eqref{reg-eq-1}}.\\
\end{aligned}
\end{equation}
Combining \eqref{reg-eq3} and \eqref{reg-eq4} implies that
\begin{equation}\label{reg-eq5}
\E |A_1(t,x)-A_1(t,y)|^{2k}\leq c_k(T)|x-y|^k,\quad\forall t\in [0,T], x,y\in[0,1], k\in\N.
\end{equation}
Similarly, using the boundedness of
{\color{blue}$\E \sup_{t\in [0,T]}\sup_{x\in[0,1]}|Z(t,x)|$}
	 and \eqref{reg-eq-2}, we obtain
\begin{equation}\label{reg-eq6}
\E |A_1(t,x)-A_1(s,x)|^{4k}\leq c_{k}(t_0,T)|t-s|^k,\quad\forall s,t\in[t_0,T],x,y\in[0,1], k\in\N.
\end{equation}
Proposition \ref{reg-prop1},  \eqref{reg-eq5}, and \eqref{reg-eq6} allow us to obtain the desired H\"older continuity with any exponent $<1/2$ in space and  any exponent $<1/4$ in time of $A_2(t,x)$.

Finally, by the Burkholder-Davis-Gundy inequality and a similar process as above (the calculation is similar to proof of Proposition \ref{s3-prop-1}), we obtain similar regularity for $\gamma(Z)(t,x)$.
In more detailed, as in Proposition \ref{s3-prop-1}, we have seen that $\gamma(Z)(t,x)\in W^{\eps,p}((0,1))$ for any $p$ and $\eps$ satisfying \eqref{cond-epsbeta}.
Moreover, $W^{\eps,p}((0,1))$ embeds continuously to the H\"older space with any exponent $\theta <\eps-1/p$. On the other hand, for any $\theta <1/2$, we can choose $p$ large and then $\eps$ satisfying \eqref{cond-epsbeta} and $\eps-1/p>\theta$.
The proof for the regularity in time for stochastic integral turns out to be similar to that of \eqref{reg-eq6} after using the Burkholder-Davis-Gundy inequality and then using \eqref{reg-eq-2}.
Because of the above H\"older continuity of $A_1(t,x),A_2(t,x)$ and $\gamma(Z)(t,x)$,
the Theorem is proved.
\end{proof}

\para{H\"older continuity on intervals containing $t=0$.}
In contrast to the case of considering compact interval in $\{t>0\}$, the regularity of $Z(t,x)$ in a compact set containing $t=0$ is more subtle.
The difficulty comes from the singularity of Neumann heat semigroup $e^{t\Delta_N}$ generated by Neumann heat kernel $G_{t}(x,y)$ when $t\downarrow 0$.
First, we have the following properties for $e^{t\Delta_N}$ on interval containing $t=0$ as follows.

\begin{lem}\label{reg-lem2}
If $z_0$ is $\alpha$-H\"older continuous, then there exists $L=L(T)>0$ such that
$$
\left|(e^{t\Delta_N}z_0)(x)-(e^{t\Delta_N}z_0)(y)\right|\leq L|x-y|^{\alpha},\quad \forall t\in [0,T], (x,y)\in [0,1],
$$
and
$$
\left|(e^{t\Delta_N}z_0)(x)-(e^{s\Delta_N}z_0)(x)\right|\leq L|t-s|^{\alpha/2},\quad\forall s,t\in [0,T], x\in [0,1].
$$
\end{lem}

\begin{proof}
A proof can be obtained by the same approach as that of Lemma \ref{reg-lem1} by using the eigenfunction expansion formula for Neumann heat kernel $G_t(x,y)$.
The reader can find similar details in \cite{Wal86}.
We provide a sketch of an alternative proof, which is interesting in its own right.
The proof is based on the relationship between heat dynamics and Brownian motion or the probabilistic solution of a PDE problem (specifically, heat equation with Neumann boundary condition).

Let us recall some facts for the case of Dirichlet boundary condition first.
Let $G^D_t(x,y)$ be the fundamental solution of heat equation on $(0,1)$ with Dirichlet boundary condition and $e^{t\Delta_D}$ be the Dirichlet heat semigroup, defined by
$$
\left(e^{t\Delta_D}u\right)(x):=\int_0^1 G_{t}^D(x,y)u(y)dy.
$$
Then $G_t^D(x,y)$ describes the transition densities of a
Brownian motion killed upon reaching $\{0, 1\}$ and for each $u$, $e^{t\Delta_D}u(x)$ will be the expectation of a functional of a Brownian motion killed on boundary given the initial condition $x$.
More precisely, if we let $B(t)$ be a standard one-dimensional Brownian motion and $\tau$ be the first time $t$ that $B(2t)$ exists $(0,1)$, then we have (see e.g., \cite[Chapter 2, Section 7]{Bas98})
\begin{equation}\label{reg-eq10}
e^{t\Delta_D}u(x)=\E_x\left(u(B_{2t});\tau<t\right),
\end{equation}
\textcolor{blue}{where $\E_x$ denotes the expectation with initial value $x$}.
 It is well known  that the fundamental solution of heat equation in $\R$ is the density of a Gaussian distribution.
When the dynamics is restricted in a domain with zero boundary condition, it should be associated with the Brownian motion killed on the boundary.
Note that the
operator $\frac12\Delta$
corresponds
to the standard Brownian motion.
That is why we have to scale the time index as above.
Because of \eqref{reg-eq10} and the (local) H\"older continuity with any exponent $<\frac 12$ of Brownian motion and $\alpha$-H\"older continuity of $z_0$, we  obtain  the conclusion in Lemma \ref{reg-lem2} for Dirichlet boundary condition case.
The detail of these arguments can be found  for example, in \cite[Proofs of Lemma 4.4 and Lemma 4.5]{Kho16}.

Coming back to our own case,
thanks to \cite[Theorem 2.5]{BCS04}, a similar expression to \eqref{reg-eq10} is obtained for the case of Neumann boundary condition by replacing the Brownian motion killed on boundary by the Brownian motion reflecting on boundary, namely, reflecting Brownian motion (RBM).
By \cite[Theorem 2.5]{BCS04}, for example, if we let $B_R$ be a one-dimensional reflecting Brownian motion in $(0,1)$ (\cite[Theorem 2.1]{BCS04} for definition), then
\begin{equation}\label{reg-eq11}
e^{t\Delta_N}u(x)=\E_x\left(u(B_R(2t)\right).
\end{equation}
Moreover, it is also noted that
$B_R(t)$ is local H\"older continuous with any exponent $<\frac 12$; see e.g., \cite[Section 2]{Tan79} and also \cite{LS84}.
Therefore, the expression \eqref{reg-eq11}, the local H\"older continuity of RBM, and the $\alpha$-H\"older continuity of $z_0$ yields the results in the Lemma as in the case of Dirichlet boundary condition.
\end{proof}

\begin{thm}\label{reg-thm2}
Assume that the non-negative initial value $Z_0$ is $\alpha$-H\"older continuous, for some $\alpha\in (0,1]$. On a compact set of time containing $t=0$, the solution $Z(t,x)$ of \eqref{eq-main} is H\"older continuous in space with exponent $<\alpha\wedge\frac 12$ and is H\"older continuous in time with exponent $<\frac{\alpha}2\wedge\frac 14$.
That is, for any $0<T<\infty$, and $\beta_1\in (0,\alpha\wedge1/2)$, $\beta_2\in (0,\alpha/2\wedge1/4)$, there is a finite random variable $C_{H}=C_H(T,\beta_1,\beta_2)$ such that
$$
|Z(t,x)-Z(s,y)|\leq C_H\left(|x-y|^{\beta_1}+|t-s|^{\beta_2}\right),\quad\forall x,y\in [0,1], s,t\in [0,T]\a.s
$$
\end{thm}

\begin{proof}[Proof of Theorem \ref{reg-thm2}]
Once we have Lemma \ref{reg-lem2} in hand, we are able to take care of the singularity of $e^{t\Delta}$ at $0$.
Hence, the proof of this Theorem is similar to that of Theorem \ref{reg-thm1} and is left to the reader.
Note that as in proof of Theorem \ref{reg-thm1},
we have already established the boundedness (on compact interval of time) of solutions on the space of continuous function with any order.

\end{proof}

\section{Existence of density}\label{sec:den}
This section is devoted to  the existence of densities of $U(t,x),V(t,x)$.
By using Malliavin calculus, we prove that
for any $t>0$, $x\in (0,1)$, $U(t,x)$ and $V(t,x)$ have absolutely continuous laws with respect to the Lebesgue measure and hence
possess densities.
It is noted that the coefficients
 for our system
 are neither Lipschitz continuous nor having linear growth.
Some notation and preliminary results in the Malliavin calculus used in this Section are given in Section \ref{sec:Amal}.

\begin{prop}\label{den-prop3}
Suppose that $(U_0,V_0)\in E$, $U_0,V_0\geq 0$.
For any $t>0$ and $x\in (0,1)$, $U(t,x)$ and $V(t,x)$ belong to $\D^{1,2}_{\text{loc}}$.
\end{prop}

\begin{proof}
In Section \ref{sec:wel}, we have already known that for any $T>0$, $p\geq 1$,
\begin{equation}\label{s6-eq-bounded}
\E \Big(\sup_{t\in[0,T]}\sup_{x\in[0,1]}|U(t,x)|^p+|V(t,x)|^p\Big)<\infty.
\end{equation}
Therefore, by Proposition \ref{den-prop1}, we can assume that the coefficients together with their derivatives are bounded
in the proof of this Proposition (with
$U(t,x),V(t,x)\in\D^{1,2}$).
In fact,
we can truncate the coefficients similarly to the process in the proof of Theorem \ref{thm-exi} such that the truncated functions are bounded together with their derivatives and coincide with the original coefficients in finite balls.

It is shown in \cite[Proposition 2.4]{PT93} that
by approximating PDEs
using finite elements method and then, approximating SDEs driven by infinite dimensional noise by SDEs driven by finite dimensional noise (alternatively, see  \cite[Section 4]{BP98}), one has that
for any $h(t,x)=\rho(t)e_i(x)$, where $\rho(t)$ is some function satisfying \textcolor{blue}{$\rho(t)\in L^2(\R_+)$},
$U(t,x),V(t,x)\in \D^h$ and
$D_hU(t,x)$,  $D_hV(t,x)$ is
the solution of the following SPDE
\begin{equation}\label{den-eq-Dh}
\begin{cases}
\displaystyle D_hU(t,x)=\int_0^t\int_0^1 G_{t-s}(x,y)\sigma_1(y)U(s,y)h(s,y)dyds\\[2ex]
\displaystyle\hspace{0.5cm}+\int_0^t\int_0^1 G_{t-s}(x,y)\Big(\frac{\partial F_1}{\partial U}(U,V)D_hU(s,y)+\frac{\partial F_1}{\partial V}(U,V)D_hV(s,y)\Big)dyds\\[2ex]
\displaystyle\hspace{0.5cm}+\int_0^t\int_0^1 G_{t-s}(x,y)\sigma_1(y)D_hU(s,y)W_1(ds,dy)\\[2ex]
\displaystyle D_hV(t,x)=\int_0^t\int_0^1 G_{t-s}(x,y)\sigma_2(y)V(s,y)h(s,y)dyds\\[2ex]
\displaystyle\hspace{0.5cm}+\int_0^t\int_0^1 G_{t-s}(x,y)\Big(\frac{\partial F_2}{\partial U}(U,V)D_hU(s,y)+\frac{\partial F_2}{\partial V}(U,V)D_hV(s,y)\Big)dyds\\[2ex]
\displaystyle\hspace{0.5cm}+\int_0^t\int_0^1 G_{t-s}(x,y)\sigma_2(y)D_hV(s,y)W_2(ds,dy).
\end{cases}
\end{equation}
It remains to show that if $\{h_k\}_{k=1}^\infty$ is an orthonormal
basis of $L^2(\R_+\times(0,1))$, then
$$
\sum_{k=1}^\infty \E\left(|D_{h_k}U(t,x)|^2\right)<\infty,\quad
\sum_{k=1}^\infty \E\left(|D_{h_k}V(t,x)|^2\right)<\infty.
$$
Using the
assumption on boundedness of $F_1,F_2$ as well as their derivatives in this Proposition, we have that
\begin{equation}\label{eq-7-27-11}
\begin{aligned}
\E \Big(&|D_{h_k}U(t,x)|^2+|D_{h_k}V(t,x)|^2\Big)\\
\leq&\; c\; \E\int_0^t\int_0^t G_{t-s}^2(x,y)\Big((D_{h_k}U(s,y))^2+(D_{h_k}U(s,y))^2\Big)dyds\\
&+c \;\E\left(\int_0^t\int_0^t G_{t-s}(x,y)\Big(U(s,y)+V(s,y)\Big)dyds\right)^2.
\end{aligned}
\end{equation}
In the above, the stochastic integral is estimated with the help of the Burkholder-Davis-Gundy inequality as in Proposition \ref{s3-prop-1}.
Let $$N_m(t):=\sup_{x\in[0,1]} \E\sum_{k=1}^m \Big(|D_{h_k}U(t,x)|^2+|D_{h_k}V(t,x)|^2\Big).$$
The boundedness of $U(s,y),V(s,y)$ developed in Section \ref{sec:wel} (see also \eqref{s6-eq-bounded}) and property \eqref{prop-A1} of $G_{t}(x,y)$ allows us to obtain from \eqref{eq-7-27-11} that
$$
\begin{aligned}
N_m(t)\leq& c\int_0^t\int_0^1 G_{t-s}^2(x,y)N_m(s)dyds+c\int_0^t\int_0^1 G_{t-s}^2(x,y)dyds\\
\leq & c\Big(1+\int_0^t \frac{N_m(s)}{\sqrt{t-s}}ds\Big)\leq c\Big(1+\int_0^tN_m(s)ds\Big),
\end{aligned}
$$
where $c$ is a finite constant, independent of $m$.
As a consequence, $N_m(t)\leq ce^{ct}, \forall m$ and thus,
$$\sup_{x\in[0,1]} \E\sum_{k=1}^\infty \Big(|D_{h_k}U(t,x)|^2+|D_{h_k}V(t,x)|^2\Big)<\infty.$$
 The proof is complete.
\end{proof}

\begin{thm}\label{den-thm1}
Suppose that the initial value $(U_0,V_0)\in E$, $U_0,V_0\geq 0$ such that $\sigma_1U_0,\sigma_2V_0\geq 0$ but not identical to $0$.
For each $t>0$ and $x\in (0,1)$, the law of $U(t,x)$ and $V(t,x)$ are absolutely continuous with respect to the Lebesgue measure.
\end{thm}

\begin{proof}
Let $t>0$ and $x\in (0,1)$ be fixed. Using a standard localization procedure,
it suffices to prove the results under the assumption that $F_1$ and $F_2$ have bounded derivatives.
We deduce from the proof of Proposition \ref{den-prop3} (or see \cite[Remark 4.2]{BP98}) that if $\theta<t$,
 $D_{\theta,\xi}U(t,x),D_{\theta,\xi}V(t,x)$ is the solution of
\begin{equation}\label{den-eq-duv}
\begin{cases}
\displaystyle D_{\theta,\xi}U(t,x)=G_{t-\theta}(x,\xi)\sigma_1(\xi)U(\theta,\xi)\\[1.5ex]
\displaystyle\hspace{0.5cm}+\int_\theta^{t}\int_0^1 G_{t-s}(x,y)\Big(\frac{\partial F_1}{\partial U}(U,V)D_{\theta,\xi}U(s,y)+\frac{\partial F_1}{\partial V}(U,V)D_{\theta,\xi}V(s,y)\Big)dyds\\[1.5ex]
\displaystyle\hspace{0.5cm}+\int_\theta^t\int_0^1 G_{t-s}(x,y)\sigma_1(y)D_{\theta,\xi}U(s,y)W_1(ds,dy)\\
\displaystyle D_{\theta,\xi}V(t,x)=G_{t-\theta}(x,\xi)\sigma_2(\xi)V(\theta,\xi)\\
\displaystyle\hspace{0.5cm}+\int_\theta^t\int_0^1 G_{t-s}(x,y)\Big(\frac{\partial F_2}{\partial U}(U,V)D_{\theta,\xi}U(s,y)+\frac{\partial F_2}{\partial V}(U,V)D_{\theta,\xi}V(s,y)\Big)dyds\\
\displaystyle\hspace{0.5cm}+\int_\theta^t\int_0^1 G_{t-s}(x,y)\sigma_2(y)D_{\theta,\xi}V(s,y)W_2(ds,dy).
\end{cases}
\end{equation}
and if $\theta>t$ , $D_{\theta,\xi}U(t,x)=D_{\theta,\xi}V(t,x)=0$.
It can be seen that
$$
\|DU(t,x)\|>0\iff \int_0^t\int_0^1 |D_{\theta,\xi}U(t,x)|d\xi d\theta>0.
$$
Let
$$
U_D(\theta;t,x):= \int_0^1 D_{\theta,\xi}U(t,x)d\xi,\quad V_D(\theta;t,x):= \int_0^1 D_{\theta,\xi}V(t,x)d\xi.
$$
Then, $U_D(\theta;t,x)$, $V_D(\theta;t,x)$ is the solutions to
\begin{equation}\label{den-eq-9}
\begin{cases}
\displaystyle U_D(\theta;t,x)=\int_0^1 G_{t-\theta}(x,y)\sigma_1(y)U(\theta,y)dy\\[1.5ex]
\displaystyle\hspace{0.5cm}+\int_\theta^{t}\int_0^1 G_{t-s}(x,y)\Big(\frac{\partial F_1}{\partial U}(U,V)U_D(\theta;s,y)+\frac{\partial F_1}{\partial V}(U,V)V_D(\theta;s,y)\Big)dyds\\[1.5ex]
\displaystyle\hspace{0.5cm}+\int_\theta^t\int_0^1 G_{t-s}(x,y)\sigma_1(y)U_D(\theta;s,y)W_1(ds,dy)\\[1.5ex]
\displaystyle V_D(\theta;t,x)=\int_0^1 G_{t-\theta}(x,y)\sigma_2(y)V(\theta,y)dy\\
\displaystyle\hspace{0.5cm}+\int_\theta^t\int_0^1 G_{t-s}(x,y)\Big(\frac{\partial F_2}{\partial U}(U,V)U_D(\theta;s,y)+\frac{\partial F_2}{\partial V}(U,V)V_D(\theta;s,y)\Big)dyds\\
\displaystyle\hspace{0.5cm}+\int_\theta^t\int_0^1 G_{t-s}(x,y)\sigma_2(y)V_D(\theta;s,y)W_2(ds,dy),
\end{cases}
\end{equation}
It is noted again that as  at the beginning of the proof of this Proposition, we can assume that $F_1$ and $F_2$ are smooth functions with bounded derivatives.
Moreover, under this assumption, as in \cite[Proposition 3.1]{PT93}
or \cite[Theorem 2]{MN08} (see also Proposition \ref{s3-prop-pos} in Section \ref{sec:wel}), one has that for any
  $t>\theta$, $x\in(0,1)$, $U_D(\theta;t,x)>0$, $V_D(\theta;t,x)>0\a.s$
Therefore, applying Proposition \ref{den-prop2} yields the desired result.
\end{proof}

\section{Existence of invariant measure}\label{sec:inv}
 From an application point of view in general and a biological point of view in particular,
the longtime behavior is one of the most important properties.
 A fundamental question in investigating the longtime behavior is whether an invariant measure exists.
This section is devoted to answering this question.

In what follows, we will consider the process $Z^z(t)$ on
$E_+:=\{(u,v)\in E: u,v\geq 0\}$ ($E_+$ is a Polish space since it is closed subset of $E$).
Non-negativity of solutions (Theorem \ref{thm-exi}) guarantees that $Z^z(t)\in E_+,\forall t\geq 0\a.s$ provided $z\in E_+$.
We first recall some notation and preliminaries as in \cite{Cer03} as follows.
Let $B_b(E_+)$ be the Banach space of bounded measurable functions $\varphi:E_+\to\R$ endowed with the sup-norm
$$
\|\varphi\|_0=\sup_{z\in \textcolor{blue}{E_+}}|\varphi(z)|,
$$
and $C_b(E_+)$ be the subspace of $B_b(E_+)$ containing continuous functions,
and $C_b^1(E_+)$ be the Banach space of differentiable functions $\varphi:E_+\to \R$ having continuous and bounded derivatives.
We define the transition semigroup $P_t$ associated with system \eqref{eq-main} as follow.
For any $z\in E_+$, $t\geq 0$, $\varphi\in B_b(E_+)$, define
$$
P_t\varphi(z):=\E\varphi(Z^z(t)),
$$
where $Z^z(t)=(U(t),V(t))$ is the solution of \eqref{eq-main} with initial condition $Z(0)=z$.

\begin{prop}
	The transition semigroup $P_t$ is Feller.
\end{prop}

\begin{proof}
It follows immediately from Proposition \ref{wel-prop-2} that for $\varphi\in C^1_b(E_+)$
$$
\begin{aligned}
|P_t\varphi(z_1)-P_t\varphi(z_2)|
&\leq \E |\varphi(Z^{z_1}(t))-\varphi(Z^{z_2}(t))|\\
&\leq \|\varphi\|_{C^1_b}\E| Z^{z_1}(t)-Z^{z_2}(t)|\to 0\text{ as }|z_1-z_2|_E\to 0.
\end{aligned}
$$
Moreover, if $\varphi\in C_b(E_+)$, we can approximate $\varphi$ (in the sup-norm) by a sequence $\{\varphi_n\}$, $\varphi_n\in C_b^1(E_+)$. Therefore, we can obtain $P_t\varphi\in C_b(E_+)$ for any $\varphi\in C_b(E_+)$ and then, the proof is complete.
\end{proof}

A probability measure $\mu$ on $(E_+,\B(E_+))$ is an invariant measure of $P_t$ if for any $t\geq 0$, $\varphi\in C_b(E_+)$,
$$
\int_{E_+} P_t\varphi(z)d\mu(z)=\int_{E_+} \varphi(z)d\mu(z).
$$
Our aim in this section is to prove that $P_t$ has an invariant measure.

The challenges come from the infinite-dimensional space, in which, a bounded set is not necessarily
relatively compact.
Since the H\"older space $C^\theta([0,1],\R^2)$ (space of $\theta$-H\"older continuous function) is compactly embedded to $E$ for any $\theta>0$, if we can prove that
\begin{equation}\label{inv-eq-test}
\sup_{t\geq t_0}\E|Z^z(t)|_{C^\theta([0,1],\R^2)}<\infty,
\end{equation}
for some $z\in E_+$, $\theta>0$, $t_0\geq 0$ then the family of (probability) measures $\{P_t(z,\cdot)\}$ is tight.
Therefore, the Krylov-Bogoliubov theorem (see e.g., \cite[Section 3.1]{PZ96}) implies the existence of an invariant measure of $P_t$.
Hence, the remaining of this Section is devoted to proof of \eqref{inv-eq-test}.

\begin{prop}\label{inv-prop2}
Assume that
\begin{equation}\label{inv-eq-1}
\inf_{x\in [0,1]}a_1(x)>0\text{ and }\inf_{x\in[0,1]}a_2(x)>0.
\end{equation}
Let $z=(U_0,V_0)\in E$, $U_0,V_0\geq 0$.
For any $p\geq 1$,
$$
\E \sup_{t\geq 0}|Z^z(t)|_E^p\leq c_p(1+|z|_E^p).
$$
\end{prop}

\begin{proof}
This Proposition is proved by applying \cite[Proposition 6.1]{Cer03}.
The validity of \cite[Proposition 6.1]{Cer03} for our system in this Proposition is shown as follows.

As mentioned earlier, the condition on growth rates of coefficients used in \cite{Cer03} is not satisfied in our setting.
However,
we have already established some ``nice" properties for the solution.
Therefore, these conditions are not needed in this section since we still guarantee necessary properties used in \cite[Proof of Proposition 6.1]{Cer03}.

There is one condition that we need to verify, which is essentially, a
``decaying outside a large ball" condition for the reaction term \cite[Condition (5.17) or (5.19)]{Cer03}, namely
\begin{equation}\label{inv-eq-cd}
\langle F(z+h)-F(z),\delta_h\rangle_{E}\leq -a|h|_E^2+b(1+|z|_E^2),\;\forall z,h\in E,
\end{equation}
for some constants $a,b>0$.
In the above, $\delta_h$ and $\langle \cdot,\delta_h\rangle_E$ are defined as follows.
For any $\delta\in E^*$ having norm 1, $\delta_h\in E^*$ defined for any $z\in E$ by
$$
\langle z,\delta_h\rangle_E:=
\begin{cases}
\frac 1{|h|_E}\sum_{i=1}^2 z_i(\xi_i)h_i(\xi_i),\text{ if }h\neq 0\\
\langle \delta,z\rangle_{E^*}\text{ if }h=0,
\end{cases}
$$
where $\xi_i\in [0,1], i=1,2$ is such that $|h_i(\xi_i)|=\max_{x\in[0,1]}|h_i(x)|$.

The original coefficient $F=(F_1,F_2)$ of \eqref{eq-main} does not satisfy this condition.
However, we make use of the non-negativity of solutions by considering
$$
F_+(U,V)(x):=F(x,U(x)\vee 0,V(x)\vee 0).
$$
and let $U_+(t)$, $V_+(t)$ be the solution to \eqref{eq-main} when $F$ is replaced by $F_+$ with the same initial condition.
Thanks to \eqref{inv-eq-1}, it is easy to verify that $F_+$ satisfies the condition \eqref{inv-eq-cd}.
Therefore, the conclusion holds for $Z_+=(U_+,V_+)$.

Finally, since the initial data $z=(U_0,V_0)$ satisfy $U_0,V_0\geq 0$, by the non-negativity of the solution in Section \ref{sec:wel}, we have $U_+(t),V_+(t)\geq 0$ for all $t\geq 0$. As a consequence, $U_+(t)=U(t)$, $V_+=V(t), \forall t\geq 0$.
Therefore, the proof is complete.
\end{proof}

\begin{prop}
Assume that $\inf_{x\in [0,1]}a_1(x)>0$ and $\inf_{x\in[0,1]}a_2(x)>0$.
Let $z=(U_0,V_0)\in E$, $U_0,V_0\geq 0$.
There are $\theta>0$ and $t_0>0$ such that
$$
\sup_{t\geq t_0}\E |Z^z(t)|_{C^\theta([0,1],\R^2)}<\infty.
$$
\end{prop}

\begin{proof}
Once we have Proposition \ref{inv-prop2}, the proof is similar to \cite[Proof of Theorem 6.2]{Cer03}.
 Because Proposition \ref{inv-prop2} has already established the uniformly bounded in space of continuous function $E$ of the solution,
the remaining task is to
obtain this property in $C^\theta$, the space of H\"older continuous function, with some sufficiently small $\theta$.
In fact, we can take care of the stochastic integral by using the Proposition \ref{s3-prop-1} and the fact $W^{\eps,p}$ is embedded into $C^\theta$, for some $\theta<\eps-1/p$.
The convolution of initial condition and the drift term can be handled by using Proposition \ref{inv-prop2} and \cite[property (2.6) and Theorem 2.6]{Cer03}.
\end{proof}

We state the results we have just proved to close this section.

\begin{thm}
Assume that $\inf_{x\in [0,1]}a_1(x)>0$ and $\inf_{x\in[0,1]}a_2(x)>0$.
The transition semigroup $P_t$ associated to $Z^z(t)$ admits an invariant measure in $E_+$.
\end{thm}

\para{Remark on uniqueness of invariant measure.}
In contrast to the existence of invariant measure, the uniqueness is more subtle.
Compared with stochastic differential equations (SDEs), the strong Feller property of the solutions of SPDEs is not easy to obtain, and
without strong Feller property
 Doob’s method
 cannot be used
to prove uniqueness of the invariant measure.
Much effort has been devoted to proving the uniqueness of invariant measure for the solutions of SPDEs in the literature.
In \cite{PZ95}, Peszat and Zabczyk proved that if the coefficients are Lipschitz continuous and the diffusion term is non-degenerate, the transition semigroup is strong Feller and irreducible, and then, the uniqueness of invariant measure is ensured.
Moreover, the class of SPDEs with non-Lipschitz and bounded drift but additive noise (constant diffusion) was
investigated in \cite{CG95}.
The reader can gain more insights by consulting the book \cite{PZ96}.
To the best of our knowledge, with non-Lipschitz coefficients and multiplicative noise, proving
the strong Feller property and/or
the uniqueness of invariant measure
still remains to be an open question.

\section{Coexistence and extinction}\label{sec:lon}

One of the most important questions studied widely in mathematical biology is whether a species under consideration is extinct or not.
Sufficient conditions for coexistence and extinction of the species in stochastic population in general and  competitive system in particular is
interesting and attractive to biologists.

This section presents some ideas and methods for this problem in our setting as well as the first attempt in providing sufficient condition for
extinction for stochastic Lotka-Volterra competitive reaction-diffusion system perturbed by space-time white noise.
The study of longtime properties of deterministic and/or stochastic populations in more simple frameworks has a long history.
An overview of that and the difficulties in our own system are discussed carefully in the Section \ref{sec:lon:dis}.

\subsection{Mild stochastic calculus}

One of the main difficulties in studying longtime properties of the system in our setting is the lack of machinery to handle the change of variables. For ODEs and PDEs, the usual calculus tools can be used. In
stochastic differential equations (SDEs)
and stochastic functional differential equations (SFDEs),  It\^o  rule and/or functional It\^o rules enable us to change the variable
relatively easily.
However,
the classical It\^o formula is no longer valid for mild solutions of SPDEs.

In this section, we  recall briefly the mild stochastic calculus and the mild It\^o formula developed recently by Da Prato, Jentzen, and R\"ocker in \cite{PJR19}; see also
 the construction and results in the paper.

Let $\check H\subseteq H\subseteq \hat H$ and $U$ be real Hilbert spaces, $W$ be cylindrical $Q$-Wiener process on $\{\Omega,\F,\{\F_t\},\PP\}$ with covariance operator $Q$ and $U_0:=Q^{1/2}(U)$, and $HS(U_0,\hat H)$ be the space of Hilbert-Schmidt operator from $U_0$ to $\hat H$.

\begin{deff}{\rm
	We say that $X$ is a mild It\^o process on $\{\Omega,\F,\{\F_t\},\PP,W,\check H,H,\hat H\}$ with evolution family $S$, mild drift $F$ and mild diffusion $G$ if and only if the followings hold
	\begin{itemize}
		\item [\rm{(i)}] $X:[0,\infty)\times \Omega \to H$ is an $\F_t$-predictable stochastic process,
		\item [\rm{(ii)}] $F: [0,\infty]\times\Omega\to \hat H$ is an $\F_t$-predictable stochastic process,
		\item [\rm{(iii)}] $G: [0,\infty)\times\Omega \to HS(U_0,\hat H)$ is an $\F_t$-predictable stochastic process,
		\item [\rm{(iv)}] $S: \{(t_1,t_2): 0\leq t_1<t_2\}\to L(\hat H,\check H)$ is a measurable function (see \cite[Section 2.1]{PJR19} for detailed construction of the $\sigma$-algebra on $L(\hat H,\check H)$) satisfying that for all $t_1<t_2<t_3$, $S_{t_2,t_3}S_{t_1,t_2}=S_{t_1,t_3}$,
		\item [\rm{(v)}] for all $t>0$, it holds \a.s that
		$$
		\int_0^t \|S_{s,t}F_s\|_{\check H}+\|S_{s,t}G_s\|^2_{HS(U_0,\check H)}ds<\infty,
		$$
		and
		$$
		X_t=S_{0,t}X_0+\int_0^t S_{s,t}F_sds+\int_0^t S_{s,t}G_sdW_s.
		$$
	\end{itemize} }
\end{deff}

\begin{thm}\label{s7-thm-mild-ito} {\rm(\cite[Theorem 1, Section 2]{PJR19})} {\rm(The mild It\^o formula).}
	Let $X:[0,\infty)\times\Omega\to H$ be a mild It\^o formula with evolution family $S:\{(t_1,t_2): 0\leq t_1<t_2\}\to L(\hat H,\check H)$, mild drift $F: [0,\infty]\times\Omega\to \hat H$ and mild diffusion $G: [0,\infty)\times\Omega \to HS(U_0,\hat H)$.
	Let $V$ be a real separable Hilbert space and $\mathbb U\subset U_0$ be an arbitrary orthonormal basis of $U_0$.
	Then, for all $\varphi\in C^{1,2}([0,\infty)\times \check H,V)$, $t_0<t\in [0,\infty)$, it holds \a.s that
	$$
	\int_0^t \left\|\frac{\partial\varphi}{\partial X}(s,S_{s,t}X_s)S_{s,t}F_s\right\|_V+\left\| \frac{\partial \varphi}{\partial X}(s,S_{s,t}X_s)S_{s,t}G_s\right\|^2_{HS(U_0,V)}ds<\infty,
	$$
	and
	$$
	\int_0^t \left\| \frac{\partial\varphi}{\partial t}(s,S_{s,t}X_s)\right\|_V+\left\|\frac{\partial^2\varphi}{\partial X^2}(s,S_{s,t}X_s)\right\|\left\|S_{s,t}G_s\right\|^2_{HS(U_0,\check H)}ds<\infty,
	$$
	and
	$$
	\begin{aligned}
	\varphi(t,X_t)=&\varphi(t_0,S_{t_0,t}X_{t_0})+\int_{t_0}^t \frac{\partial\varphi}{\partial t}(s,S_{s,t}X_s)ds+\int_{t_0}^t \frac{\partial\varphi}{\partial X}(s,S_{s,t}X_s)S_{s,t}F_sds\\
	&+\frac 12\sum_{u\in\mathbb U}\int_{t_0}^t\frac{\partial^2\varphi}{\partial X^2}(s,S_{s,t}X_s)(S_{s,t}G_su,S_{s,t}G_su)ds\\
	&+\int_{t_0}^t \frac{\partial \varphi}{\partial X}(s,S_{s,t}X_s)S_{s,t}Z_sdW_s.
	\end{aligned}
	$$
\end{thm}

\subsection{A first result}
In this section, we provide a first result on sufficient condition for the extinction (and equivalently, of course, necessary conditions for permanence) of the Lotka-Volterra competitive model in SPDEs setting.

\begin{thm}\label{lon-ext-thm1}
	Assume that $\sup_{x\in [0,1]}m_1(x)<\frac 12\inf_{x\in[0,1]}\sigma_1^2(x)$.
	For any initial $(U_0,V_0)\in E$, $U_0,V_0\geq 0$, one has that
	$$
	\limsup_{t\to\infty}\E \ln\int_0^1 U(t,x)dx=-\infty.
	$$
	Similarly, if $\sup_{x\in [0,1]}m_2(x)<\frac 12\inf_{x\in[0,1]}\sigma_2^2(x)$ then
	$$
	\limsup_{t\to\infty}\E \ln\int_0^1 V(t,x)dx=-\infty.
	$$
\end{thm}

\begin{proof}
	For arbitrarily fixed $\eta>0$,
	directed calculations show that at $v\in L^2((0,1),\R)$ satisfying $v\geq 0$, the first and second Fr\'echet derivative of the functional $\varphi_\eta(v):=\ln \left(\int_0^1 v(y)dy+\eta\right)$, denoted by $\frac{\partial \varphi_\eta}{\partial X}(v)$ and
	$\frac{\partial^2\varphi_\eta}{\partial X^2}(v)$, are  as follows
	$$
	\frac{\partial \varphi_\eta}{\partial X}(v)h=\frac{\int_0^1 h(y)dy}{\eta+\int_0^1 v(y)dy},\;h\in L^2((0,1),\R),
	$$
	and
	$$
	\frac{\partial^2\varphi_\eta}{\partial X^2}(v)(h_1,h_2)=-\frac{\int_0^1 h_1(y)dy\int_0^1h_2(y)dy}{\left(\eta+\int_0^1 v(y)dy\right)^2},\;h_1,h_2\in L^2((0,1),\R).
	$$
	By the mild It\^o formula (see  Theorem \ref{s7-thm-mild-ito}), we have that
	\begin{equation}\label{ext-eq-1}
	\begin{aligned}
	\ln&\Big(\eta+ \int_0^1 U(t,x)dx\Big)\\
	=&\ln \Big(\eta+\int_0^1 \left(e^{t\Delta_N} U_0\right)(x)dx\Big)\\
	&+\int_0^t\frac{\int_0^1 \big(e^{(t-s)\Delta_N}U(s)(m_1-a_1U(s)-b_1V(s))\big)(x)dx}{\eta+\int_0^1 \left(e^{(t-s)\Delta_N}U(s)\right)(x)dx}ds\\
	&-\frac 12\int_0^t \sum_{k=1}^\infty \frac{\left(\int_0^1 \big(e^{(t-s)\Delta_N}U(s)e_k\sigma_1\big)(x)dx\right)^2}{\left(\eta+\int_0^1 e^{(t-s)\Delta_N}U(s)(x)dx\right)^2}+\int_0^t J_\eta(t,s)dW_1(s),
	\end{aligned}
	\end{equation}
	where $J_\eta(t,s)$ is linear operator from $L^2((0,1),\R)$ to $\R$ defined by
	$$
	J_\eta(t,s)(h):=\frac{\int_0^1 \big(e^{(t-s)\Delta_N}U(s)h\sigma_1\big)(x)dx}{\eta+\int_0^1 \big(e^{(t-s)\Delta_N}U(s)\big)(x)dx},\;h\in L^2((0,1),\R).
	$$
	Set
	\begin{equation}\label{ext-eq-2}
	M_\eta(U,t,s)=\sum_{k=1}^\infty \frac{\Big(\int_0^1 \big(e^{(t-s)\Delta_N}U(s)e_k\big)(x)dx\Big)^2}{\Big(\eta+\int_0^1 e^{(t-s)\Delta_N}U(s)(x)dx\Big)^2}.
	\end{equation}
	Parseval's identity and H\"older's inequality show that
	\begin{equation}\label{ext-eq-3}
	\begin{aligned}
	\sum_{k=1}^\infty& \Big(\int_0^1 \big(e^{(t-s)\Delta_N}U(s)e_k\big)(x)dx\Big)^2\\
	&=\sum_{k=1}^\infty \Big(\int_0^1\int_0^1 G_{t-s}(x,y)U(s,y)e_k(y)dydx\Big)^2\\
	&=\sum_{k=1}^\infty
	\Big\langle \int_0^1 G_{t-s}(x,\cdot)U(s,\cdot)dx,e_k(\cdot)\Big\rangle_{L^2((0,1),\R)}^2\\
	&=\int_0^1 \Big(\int_0^1 G_{t-s}(x,y)U(s,y)dx\Big)^2dy\\
	&\geq  \Big(\int_0^1\int_0^1 G_{t-s}(x,y)U(s,y)dxdy\Big)^2.
	\end{aligned}
	\end{equation}
	We deduce from \eqref{ext-eq-2} and \eqref{ext-eq-3}  that
	\begin{equation}\label{ext-eq-4}
	\lim_{\eta\to0}M_\eta(U,t,s)\geq 1.
	\end{equation}
	It is seen that
	\begin{equation}\label{ext-eq-5}
	\frac{\disp\int_0^1 \big(e^{(t-s)\Delta_N}U(s)(m_1-a_1U(s)-b_1V(s))\big)(x)dx}{\disp\eta+\int_0^1 \left(e^{(t-s)\Delta_N}U(s)\right)(x)dx}\leq \sup_{x\in [0,1]}m_1(x).
	\end{equation}
	Taking expectation to \eqref{ext-eq-1} and then applying \eqref{ext-eq-5} imply that
	\begin{equation}\label{ext-eq-6}
	\begin{aligned}
	\E\ln\left(\eta+ \int_0^1 U(t,x)dx\right)\leq& \ln \left(\eta+\int_0^1 U_0(x)dx\right)\\
	&+\int_0^t \left(\sup_{x\in [0,1]}m_1(x)-
	\frac 12\inf_{x\in[0,1]}\sigma_1^2(x)M_\eta(U,t,s)\right)ds.
	\end{aligned}
	\end{equation}
	Letting $\eta\to0$ in \eqref{ext-eq-6} and applying \eqref{ext-eq-4}, we get
	$$
	\E\ln\int_0^1U(t,x)dx\leq \ln\int_0^1U_0(x)dx+Rt,
	$$
	where
	$$
	R:=\sup_{x\in[0,1]}m_1(x)-\frac 12\inf_{x\in[0,1]}\sigma_1^2(x)<0.
	$$
	As a consequence,
	$$
	\limsup_{t\to\infty}\E\ln\int_0^1U(t,x)dx=-\infty.
	$$
	Similarly, the results for $V(t,x)$ are also obtained.
	The proof is complete.
\end{proof}

\para{Remark on other estimates.} 
Let us comment on the difficulty in providing estimates in probability one.
For example, one may expect that the conclusion in Theorem \ref{lon-ext-thm1} is replaced by
$$
\PP\left(\limsup_{t\to\infty}\int_0^1 U(t,x)dx=0\right)=1,
$$
and
$$
\PP\left(\limsup_{t\to\infty}\int_0^1 V(t,x)dx=0\right)=1.
$$
In fact, in \cite{NY19-cosa}, we used the following Lemma, whose proof is in \cite[Lemma 4.2]{NY19-cosa} and
obtained
some results of probability one estimates
for SIS epidemic model.

\begin{lem}\label{exp}
	Let $\Phi(s)$ be $L(U,\R)$-valued process and $W$ be a (finite trace)
	$Q$-Wiener process such that
	$\int_0^t\Phi(s)dW(s)$ is well defined for any $t\geq 0$
	and $a,b$ be two positive real numbers. We have the following estimate
	$$\PP\left\{\abs{\int_0^t\Phi(s)dW(s)}-\dfrac a2\int_0^t \norm{\Phi(s)}_{HS(U,\R)}^2ds<b,\forall t\geq 0\right\}\geq 1-e^{-ab}.$$
\end{lem}

However, in contrast to the strong solution,
where the stochastic integral is in fact a martingale,
for mild solution, this result is no longer valid since the stochastic convolution $\int_0^t e^{(t-s)\Delta_N}\Phi(s)dW(s)$ is not a martingale with respect to $t$. Moreover, it is noted that we are dealing with cylindrical Wiener processes rather than (finite trace) $Q$-Wiener processes.

\subsection{Discussion}\label{sec:lon:dis}
Much effort has been placed on the study of longtime behavior of biological model in general and competitive model in particular.
Let us review some important methods, ideas and results in the literature.
At the beginning, the dynamics of individuals in the environment are usually modeled by original differential equations (ODEs).
The characterization of long-term properties is often obtained by using Lyapunov functional method, see e.g., \cite{HS98,Mur02}.
To capture the random factors, the stochastic terms are added into ODEs and turn out to study stochastic differential equations (SDEs).
In contrast to  numerous papers that used Lyapunov function method to analyze the underlying systems with limited success,
Chesson and Ellner \cite{CE89},
Schreiber and Bena\" im \cite{SBA11}
 initiated the study
by examining the
corresponding
boundary behavior and considered the stochastic rate of growth.
This idea is applied and developed by Nguyen and Yin \cite{DY17} to obtain the characterization of coexistence and extinction for Lotka-Volterra competitive equation modeled by SDEs; and then Hening and Nguyen generalized the results for a general Kolmogorov model in \cite{HN18} and Bena\"im \cite{Ben14} established a general abstract theory for this kind problem.
The readers can consult \cite{DDN19,
	 DN18,
	 DN17,
	 HN18,NNY19,NYZ20,
	 SBA11}  the references therein for works on biological and ecological models under the SDE framework.

Very recently, a class of functional SDEs model was considered by  Nguyen, Nguyen, and Yin in \cite{NNY19-2,NNY19-3}, which allows the dynamics depend on the past history.
By combining the ideas in SDEs  (considering the growth rate), techniques in SDEs in infinite dimensions, and new developed theory in functional analysis (the functional It\^o formula), the authors were able to provide sufficient and almost necessary condition for persistence and extinction with
 applications to Lotka-Volterra competitive system in stochastic functional  differential equation setting; see \cite[Section 4.1]{NNY19-2}.

All of above references assume the densities to be homogeneous in the state (or location) variable.
The inhomogeneous case needs to be considered.
One of the first attempts in studying this situation is to embed them into PDEs framework and often is known with the name ``reaction-diffusion" system. Note that the word ``diffusion" here indicates the diffusion of dynamics in space,
not
the diffusion driven by noise as in the SDEs and in fact, it is still non-random system.
The coexistence state of Lotka-Volterra competitive reaction-diffusion is investigated by Gui and Lou in \cite{GL94}.
In this work, the authors provided sufficient conditions for uniqueness and non-uniqueness of coexistence of states.
One of the most effective theories and technique in investigating the coexistence and extinction of a population in PDEs setting introduced by Wang and Zhao in \cite{WZ12} is to consider the problem for equilibrium solution and its eigenvalues.
Similar idea and theory is also applied and developed in \cite{CLL17,SLY19} to characterize the longtime behavior of epidemic reaction-diffusion equation.
In addition, there are also important works on Lotka-Volterra competitive reaction-diffusion equation in \cite{HN13,HN13-2,HN16,HN16-2,HN17,LN12,LZZ19}.

In contrast to the existing works, our model takes care both of the spatial inhomogeneity and the random factor and hence, we must study them in the SFDEs frameworks.
Unfortunately, all of ideas, methods,
and machinery
in calculations in the literature fail to be applicable to
obtain sufficient conditions
 and not to mention
 sharp condition for coexistence and extinction.
 At this moment, it does not seem that we can
the
growth rate as the indicator in the SDE models to characterize extinction and persistence.
This mainly due to the dependence of the models on the space variables.
The theory using eigenvalues of equilibrium equation is failed to be applicable here due to the appearance of stochastic noises.
In
 general,
 using the chain rule seems to be unavoidable.
However, the chain rule for mild solutions of SPDEs is more subtle and cannot be applied
effectively.

In the previous section,
we have tried to overcome the second difficulty by applying newly developed tools in stochastic calculus, namely, the mild It\^o formula and obtain sufficient conditions for extinction.
However, due to the
lack of a strong and effective abstract theory, we have not been able to provide a sharp condition.

\para{Why is the ``growth rate method" in SDEs no longer works?}
The growth rate idea is the most effective to characterize the persistence and extinction of a stochastic population modeled by SDEs; see \cite{Ben14,
	CE89,DN18,
	HN18,NNY19,
	 SBA11, TN20,TNY20} and the reference therein.
The main idea is to define the growth rate of a species using its Lyapunov exponent.
If the growth rate is positive, the number of this species will increase and
thus the population will never be extinct or it will be persistent.
Conversely, when the growth rate is negative,
they will be extinct exponentially fast in the long run.

However, the growth rate
appears  not to be able to characterize the longtime behavior for the SPDE cases.
Intuitively, the dynamics of the population of the species depend not only on the time but also on the space variable.
As a consequence, even the growth rate of a population is positive at some location $x$, the population at $x$ can still tend to $0$ since they can diffuse (in space) to their neighbors.
Similarly, in case
 the growth rate at $x$ is negative, the population at location $x$
 can still
 be persistent since the individuals may return to  the neighbors
 infinitely often if certain conditions hold.
The key is that
dynamics of populations in SPDEs setting depend on the time and space simultaneously while the ``growth rate" is only able to characterize the behavior in ``time flow".

\para{Why is the ``eigenvalue method" in PDEs no longer working?}
There is a nice idea in PDEs to study the asymptotic stability and hence, investigate the longtime property. It considers the equilibrium problem and the associated eigenvalues; see e.g., \cite{CLL17,SLY19,WZ12} as well as \cite{HN13,HN13-2,HN16,HN16-2,HN17,LN12,LZZ19}.
The equilibrium problem is defined
with
the time variable being frozen.
Roughly speaking, the solutions of PDEs will tend to the equilibria (functions independent of time variable $t$). Hence, the eigenvalues will play some role in studying the stability.
In SPDEs setting, it is not clear how to have a similar ``equilibrium problem" like PDEs case since the stochastic integral with respect to space-time white noise does not work the way as the Lebesgue integral and/or Bochner integral did.
If we integrate over $dx$ for fixed $t$, the integral can be viewed as a Bochner integral while over $dt$ for fixed $dx$,  can be viewed as an It\^o integral.
However, as given
in the appendix, the stochastic integral with respect to space-time white noise requires to integrate over space variable $dx$ and time variable $dt$ simultaneously.
Hence, the problems in SPDEs turn out to be much different compared with PDEs at this point.

\para{Approximation by strong solutions:}
There is also another approach to overcome the second difficulty (being lack of tools regarding change variable), which is introduced in our early works in \cite{NNY18,NY19-cosa,NY20}.
The idea is to approximate the mild solution by a sequence of strong solutions (e.g., the solutions corresponding to the stochastic differential equation driving by finite dimensional noise) and then, we work on these strong solutions, for which the classical It\^o's formula
is valid.
However, this method does not work  well for cylindrical Wiener process (having infinite trace).
Moreover, the convergence of the sequence of strong solutions to the mild solution is in expectation and in $L^2((0,1))$-norm.
That will be not useful in some estimates.

\para{What do we expect?} As was mentioned, our results in this section are not
sharp compared with our results in SDEs case or even SFDEs case.
Formally, we expect to introduce a Hypothesis (E) such that under (E),
$$
\limsup_{t\to\infty}\sup_{x\in [0,1]}U(t,x)=0,\quad \limsup_{t\to\infty}\sup_{x\in [0,1]}V(t,x)=0,
$$
in some sense (almost surely or in expectation or in probability);
and a Hypothesis (C) such that under $C$,
$$
\liminf_{t\to\infty}\inf_{x\in [0,1]}U(t,x)>\delta,\quad \liminf_{t\to\infty}\inf_{x\in [0,1]}V(t,x)>\delta,
$$
for some positive constant $\delta$ (independent of the initial value) in some sense (almost surely or in expectation or in probability).
Moreover, the Hypotheses (E) and (C) cover almost all possible cases and only critical
cases are left.

To obtain this sharp condition may require developing  from two different angles.
The first one is an abstract theory to characterize the longtime behavior of a stochastic population in both of ``space flow" and ``time flow" in SPDEs setting. The second one is a useful tool to
derive estimates
using the stochastic mild calculus more  effectively.

\section{High-dimensional problems}\label{sec:hig}
One of problems of SPDEs is the trade off of the dimension and the ``regularity" of the noise.
By the phase ``trade off", we mean that
the higher dimension
 one consider, the more regularity the noise needs.
To handle the problem in general Euclidean space $\R^d$ with $d>1$, we can ``inject" color into the noise and replace space-time white noise by a noise, which is white in time but color in space.
But how much ``color" we need to inject into the noise?

The easiest case is to use the finite-trace $Q$-Wiener process and we refer this case as ``nuclear case".
 Such cases were also considered in some our works in \cite{NNY18,NY19-cosa,NY20} for epidemic model and predator-prey models.
In fact, we considered the ``nuclear case" in order to simplify the arguments and help us in investigating the longtime property (sufficient conditions for persistence and extinction).
The ``nuclear case" is
more advantageous
for
approximating mild solutions by sequence of strong solutions and in estimating some quantity like ``$\ln\int (\cdots)$''; see \cite{NNY18,NY19-cosa,NY20} for the details.
However, the ``finite trace" assumption is too strong and unnecessary in some problems. We will consider problem of reducing this condition in high dimension case.

\para{Extending our work to higher-dimensional spaces.}
Now, we will illustrate the extension of some of our results [well-posedness of the problems and longtime behavior (existence of invariant measure)]
to high-dimensional space, i.e., the domain $(0,1)$ of space variable $x$ is replaced by $\0\subset \R^d$, where $\0$ is a bounded domain (having smooth boundary) of $\R^d$ with $d\geq 1$.
In the case $d>1$, we will not require the Wiener process be ``nuclear" and we will clarify how much color
is needed for
the Wiener process.

We reconstruct the noise, the driving force in our system as follows.
For simplicity of notation, we only consider the case for $W_1$ only (it is denoted by $W$ for notational simplicity), which
is the driving noise for the first equation. The case $W_2$ is similar.
Let $\{\beta_k\}_{k=1}^\infty$ be an independent sequence of $\{\F_t\}_{t\geq 0}$-adapted one-dimensional Wiener processes and $\{e_k\}_{k=1}^{\infty}$ be a complete and uniformly bounded orthonormal system in $L^2(\0,\R)$.
We define the cylindrical $Q$-Winner process $W(t)$
in  \eqref{eq-main} as follows
$$\displaystyle W(t)=\sum_{k=1}^{\infty}\lambda_ke_k\beta_{k}(t),$$
where $\{\lambda_k\}$ is a sequence of real positive numbers and $\{e_k\}$ is a complete
orthonormal system of $L^2(\0)$ of eigenfunctions of $A$, the realization of Laplace operator endowed with the Neumann condition in $L^2(\0)$, and $\{e_k\}$ is assumed to be equibounded in $L^\infty(\0)$. [Unlike the one dimension case, the property that $\{e_k\}$ is equibounded may fail in higher dimension for general domain (see \cite[Remark 2.2]{Cer03}), so we need to assume that in this Section.]
The following hypothesis (see \cite[Hypothesis 1]{Cer03}) is the answer to the question ``how much color we need for the noise."

\begin{hyp}\label{hig-hyp}
If $d=1$ then
$$
\sup_{k}\lambda_k<\infty.
$$
If $d\geq 2$, then
$$
\sum_{k=1}^\infty |\lambda_k|^p<\infty,
$$
for some
$$
2<p<\frac{2d}{d-2}.
$$
\end{hyp}

Note that for $p<q$,
$$
\sum_{k=1}^\infty |\lambda_k|^p<\infty\Rightarrow \sum_{k=1}^\infty |\lambda_k|^q<\infty,
$$
and for $p=2$, the condition turn out to be finite-trace condition.

\para{Extension 1:} Under  hypothesis \ref{hig-hyp}, our results (Theorem \ref{thm-exi} and Proposition \ref{wel-prop-2}) in Section \ref{sec:wel} still hold.
The reader can prove that by modifying  Proposition \ref{s3-prop-1}, specially \eqref{s3-prop1-eq1} and \eqref{s3-prop1-eq2}.
In Proposition \ref{s3-prop-1}, $\alpha, p,\eps$ will  also be chosen to satisfy
\begin{equation*}
\frac dp<\alpha<\frac 14\quad\text{and}\quad \frac dp<\eps<2\big(\alpha-\frac dp\big).
\end{equation*}
The general abstract computations and results can be found in \cite[Section 3]{Cer03}.
Once we have the analogous Proposition \ref{s3-prop-1}, we can mimic the remaining of the Section \ref{sec:wel}.
It is noted again that our coefficients do not satisfy the ``growth rate" condition in \cite[Theorem 5.2]{Cer03},
but we can still overcome the difficulty by a similar technique as we did in the one-dimensional case.

\para{Extension 2:}
Under  Hypothesis \ref{hig-hyp}, our result (existence of invariant measure) in Section \ref{sec:inv} still holds.
In fact, once
the results
in Section \ref{sec:wel} are valid for high-dimensional spaces,
the arguments in Section \ref{sec:inv} are almost the same.
Note that in Proposition \ref{s3-prop-1}, $\alpha, p,\eps$ will be chosen again to satisfy
$d/p<\eps$
such that $W^{\eps,p}(\0)$ is embedded into $C^\theta(\bar \0)$ for some $\theta<\eps-d/p$, and then $C^\theta(\bar \0)$ is embedded compactly into $C(\bar \0)$.

\section{Conclusion}\label{sec:con}
This work focuses on stochastic Lotka-Volterra competitive reaction-diffusions perturbed by space-time white noise.
Our proposed model stems
 from biological and ecological points of view.
The analysis is then provided for both
the mathematical problem and applications.

The dynamics of population are modeled by a SPDEs with non-Lipschitz coefficients and multiplicative noise.
Important properties including well-posedness, regularity of the solution,
 existence of density, existence of invariant measure, as well as the longtime behavior (coexistence and extinction) of Lotka-Volterra competitive reaction-diffusion systems are addressed.
The results  are also extended to higher space dimensional systems by coloring the noise.


\section{Appendix: Background materials}
\label{appendix}

The next three sections are devoted to constructions and comparisons of the two different approaches (infinite-dimensional integration theory of Da Prato and Zabczyk and random field approach of Walsh) and their equivalence in certain classes of SPDEs.
The reader can find the full construction of Walsh's theory in \cite{Wal86}, and that of  Da Prato and Zabczyk in \cite{PZ92}. A comparison of these two approaches can be found in \cite{DQ11,PZ92}.

\subsection{Infinite-dimensional integration theory}\label{sec-21}
This section provides the
formulation of a space-time white noise
 driving process in our stochastic systems together with the corresponding stochastic integral with respect to a
 standard cylindrical $Q$-Wiener process.
   First, we start with a finite trace $Q$-Wiener process.

\def\Tr{\hbox{Tr\;}}
\begin{deff} {\rm
Let $V$ be a separable Hilbert space endowed with the inner product $\langle \cdot,\cdot\rangle_V$, and $Q$ be a linear, symmetric (self-adjoint), non-negative definite, and bounded operator on $V$ such that $\Tr Q<\infty$. A $V$-valued stochastic process $\{W(t), t\geq 0\}$ is
a $Q$-Wiener process if
\begin{itemize}
\item $W(0)=0$, $W$ has continuous trajectories, and $W$ has independent increments.
\item The law of $W(t)-W(s)$ is Gaussian with mean zero and covariance operator $(t-s)Q$. That is,
for
any $h\in V$ and $0 \leq  s\leq  t$, the real-valued random variable $\langle W_t-W_s, h\rangle_V$ is Gaussian with mean zero and variance $(t-s)\langle Qh, h\rangle_V$.
\end{itemize} }
\end{deff}

Let $\{e_k\}_{k=1}^\infty$ be a complete orthonormal system in the Hilbert space $V$ such that $Qe_k=\lambda_ke_k$, where $\lambda_k$ is the strictly positive $k^{\text{th}}$ eigenvalue of $Q$ corresponding to the eigenvector $e_k$.
If we define $\tilde\beta_k(t)=\langle W(t),e_k\rangle_V$, for $t\geq 0$, $k\in\N$, and $\beta_k(t)=\frac{\tilde\beta_k(t)}{\sqrt{\lambda_k}}$, then it can be seen that $\{\beta_k(t)\}_{k=1}^\infty$ is a sequence of independent, standard, one-dimensional $\{\F_t\}$-Brownian motions, and $$W(t)=\sum_{k=1}^\infty \langle W(t),e_k\rangle_Ve_k
=\sum_{k=1}^\infty\sqrt{\lambda_k}\beta_k(t)e_k.$$
Conversely, given a sequence of independent standard Brownian motions $\{\beta_k(t)\}_{k=1}^\infty$, and a sequence $\{\lambda_k\}_{k=1}^\infty$ of positive numbers satisfying that $\sum_{k=1}^\infty \lambda_k<\infty$, we can obtain a $Q$-Wiener process $W$ by defining
$$
W(t):=\sum_{k=1}^\infty\sqrt{\lambda_k}\beta_k(t)e_k.
$$

\begin{deff}\label{def-cylin}{\rm
Let $Q$ be a symmetric (self-adjoint) and non-negative definite bounded linear operator on the Hilbert space $V$.
A family of random variables $B=\{B_t(h), t\geq 0, h\in V\}$ is a cylindrical Wiener process on $V$ if the following conditions are satisfied:
\begin{itemize}
\item for any $h\in V$, $\{B_t(h), t\geq 0\}$ is a Brownian motion with covariance $t\langle Qh,h\rangle_V$;
\item for all $s,t\geq 0$, and $h,g\in V$,
$$
\E \left(B_s(h)B_t(g)\right)=(s\wedge t)\langle Qh,g\rangle_V.
$$
\end{itemize}
We name $Q$ the covariance of $B$.
If $Q$ is the identity operator in $V$, then we call $B$ a standard cylindrical Wiener process.}
\end{deff}

Similarly, if we let $\{e_k\}_{k=1}^\infty$ be a complete orthonormal system in $V$, $B_t(h)$ be a standard cylindrical Wiener process and set $\beta_k(t):=B_t(e_k)$, then $\{\beta_k(t)\}_{k=1}^\infty$ is a sequence of independent, standard, one-dimensional Brownian motions.
Conversely, given a sequence of independent real-valued standard Brownian motions $\{\beta_k(t)\}_{k=1}^\infty$,
$$
B_t(h):=\sum_{k=1}^\infty\beta_k(t)\langle e_k,h\rangle_V,
$$
defines a standard cylindrical Wiener process in $V$.

If $\{W(t),t\geq 0\}$ is a $Q$-Wiener process on $V$, we can associate it to a cylindrical Wiener process in the sense of Definition \ref{def-cylin} by setting $B_t(h)=\langle W_t,h\rangle_V$ for any $h\in V$, $t\geq 0$. Conversely,
one may imagine that any cylindrical Wiener process is associated to a $Q$-Wiener process on a Hilbert space.
Unfortunately, this is not true in general.
 In fact, if $V$ is an infinite dimensional space, there is no $Q$-Wiener process $W$ associated to a given standard cylindrical Wiener process $B$; see \cite[Theorem 3.2]{DQ11}.
However, it is possible to construct a Hilbert-space-valued Wiener process in a larger Hilbert space $V_1$, which is associated to $B$ (in certain sense), and which will be called a cylindrical $Q$-Wiener process.
The construction is  as follows.
Let $V$ be a Hilbert space and $Q$ be a symmetric non-negative definite and bounded operator on $V$ with possibly $\Tr
Q=\infty$.
Let $\{e_k\}_{k=1}^\infty$ be an complete orthonormal system of $V$ that contains eigenvectors of $Q$ with respect to eigenvalues $\{\lambda_k\}_{k=1}^\infty$.
Define $V_0:=Q^{1/2}(V)$ as a subspace of $V$ endowed with the inner product
$$
\langle h,g\rangle_{V_0}:=\langle Q^{-1/2}h,Q^{-1/2}g\rangle_V,
$$
where $Q^{-1/2}$ is the pseudo-inverse of the operator $Q^{1/2}$.
Then, $V_0$ is also a Hilbert space.
As in \cite[Remark 2.5.1]{PR07},
it is always possible to find a Hilbert space $V_1$ such that $V$ is embedded continuously into $V_1$ and the embedding of $V_0$ into $V_1$ is Hilbert-Schmidt, i.e., there is a bounded linear injective operator $J: V\to V_1$ such that the restriction $J_0:=J_{|V_0}: V_0\to V_1$ is a Hilbert-Schmidt operator. Recall that the operator $T:V\to H$ is  Hilbert-Schmidt if for some (and then all) complete orthonormal system $\{e_k\}_{k=1}^\infty$ of $V$,
\begin{equation}\label{eq-cylin-hs}
\sum_{k=1}^\infty \|T(e_k)\|_H^2<\infty.
\end{equation}
Let  $J_0^*$ be the adjoint of $J_0$ and $Q_1:=J_0J_0^*.$

\begin{prop} $(${\rm\cite[Proposition 4.11]{PZ92}} or {\rm\cite[Proposition 3.6]{DQ11}}$)$
The formula
\begin{equation}\label{eq-cylin-1}
W(t)=\sum_{k=1}^\infty \beta_k(t)\tilde e_k,
\end{equation}
where $\{\tilde e_k\}_{k=1}^\infty$ $(\tilde e_k=Q^{1/2}(e_k))$ is a complete orthonormal system in $V_0$ and $\{\beta_k(t)\}_{k=1}^\infty$ is a sequence of independent real-valued standard Wiener processes, defines
a $Q_1$-Wiener process on $V_1$ with $\Tr Q_1<\infty$. More precisely, this $Q_1$-Wiener process has the following form in $V_1$:
\begin{equation*}
W(t)=\sum_{k=1}^\infty J_0(\tilde e_k)\beta_k(t),
\end{equation*}
\end{prop}

\begin{deff}
{\rm The process $W(t)$ defined in \eqref{eq-cylin-1} is called a  cylindrical $Q$-Wiener process if $\Tr Q=\infty$ and standard cylindrical $Q$-Wiener process if $Q$ is the identity operator.}
\end{deff}

Let $L(V,H)$ be the space of linear (not necessarily bounded) operators from $V$ to $H$,
and $L_2^0:=HS(V_0,H)$, the Hilbert space of all Hilbert-Schmidt operators from $V_0:=Q^{1/2}(V)$ into $H$ equipped with the inner product
$$
\langle \Phi,\Psi\rangle_{L_2^0}:=\sum_{k=1}^\infty \langle \Phi \tilde e_k,\Psi\tilde e_k\rangle_H,
$$
where $\{\tilde e_k\}_{k=1}^\infty$ is a complete orthonormal system of $V_0$.

For $\Phi=\{\Phi(s): s\in [0,T]\}$ being a measurable $L_2^0$-valued process satisfying
$$
\|\Phi\|_T:=\left[\E \left(\int_0^T \|\Phi(s)\|_{L_2^0}ds\right)\right]^{1/2}<\infty,
$$
the stochastic integral with respect to the cylindrical $Q$-Wiener process,
$$
\int_0^t \Phi(s)dW(s),
$$
 is constructed as follows.
First, the stochastic integral $\int_0^t \Phi(s)dW(s)$, where $W(t)$ is a $Q$-Wiener process with $\Tr Q<\infty$, is defined through the class of simple functions and then using isometry property, the details of this construction can be found in \cite[Chapter 4]{PZ92}. Now, if $Q$ is the identity operator or in general, $\Tr Q=\infty$, as in the above construction, there are a Hilbert space $V_1$ and an operator $J$ such that the restriction $J_0$ of $J$ in $V_0$ is Hilbert-Schmidt and $W(t)$ is a $Q_1$-Wiener process on $V_1$ with Tr $Q_1<\infty$.

\begin{deff}{\rm As in \cite[Proposition 3.6]{DQ11} or \cite[Proposition 4.11]{PZ92} or \cite[Proposition 2.5.2]{PR07}, we have
$$
\Phi\in L_2^0=L_2(V_0,H) \iff \Phi\circ J_0^{-1}\in L_2(Q_1^{1/2}(V_1),H).
$$
Hence, the $H$-valued stochastic integral $\int_0^t \Phi(s)dW(s)$ with respect to the cylindrical $Q$-Wiener process is defined by
$$
\int_0^t \Phi(s)dW(s):=\int_0^t \Phi(s)\circ J_0^{-1}dW(s),
$$
where the integral on the right-hand side is the integral with respect to the (finite trace) $Q_1$-Wiener process defined in $V_1$ previously.
 Note that the above definition does not depend on the choice of space $V_1$.}
\end{deff}

Now, let $B$ be a cylindrical Wiener process in $V$ (definition \ref{def-cylin}) and $V_Q$ be the Hilbert space $V$ equipped with the inner product
$$
\langle h,g\rangle_{V_Q}:=\langle Qh,g\rangle_V,\quad h,g\in V,
$$
 $\{v_k\}_{k=1}^\infty$ be a complete orthonormal system of $V_Q$, and $g\in L^2(\Omega\times [0,T];V_Q)$ be predictable process. We define the integral $\int_0^ Tg(s)dB(s)$ as follows
$$
\int_0^T g(s)dB(s):=\sum_{k=1}^\infty \int_0^T\langle g(s),v_k\rangle_{V_Q}dB_s(v_k).
$$
Moreover, we can associate $B$ to a cylindrical $Q$-Wiener process defined by \eqref{eq-cylin-1} with $\beta_k(t)=B_t(e_k)$, $\{e_k\}_{k=1}^\infty$ is a basic of $V$. Then, the above stochastic integrals are connected in
 the following Proposition.

\begin{prop}\label{inf-prop-1}$(${\rm\cite[Proposition 3.9]{DQ11}}$)$
Define $\Phi^g_s:V\to \R$ by $\Phi_s^g(\eta)=\langle g(s),\eta\rangle_V$. Then  $\{\Phi_s^g, s\in[0,T]\}$ is a predictable process with value in $L_2(V_0,\R)$,
$$
\E\left(\int_0^T \|\Phi_s^g\|_{L_2}^2\right)=\E\left(\int_0^T \|g(s)\|_{V_Q}^2ds\right)
,$$
and
$$
\int_0^T \Phi^g_sdW(s)=\int_0^T g(s)dB(s).
$$
\end{prop}

\begin{deff}\label{inf-def-1}
{\rm With $\Phi^g_s$ being defined as in Proposition \ref{inf-prop-1},
 define
$$
\int_0^t \left\langle g(s),dW(s)\right\rangle_{V}:=\int_0^T \Phi^g_sdW(s).
$$}
\end{deff}

\subsection{Random field approach}\label{sec-22}
In this section, we recall some definitions of space-time white noise and random field approach introduced by Walsh.
We  discuss these briefly only for our own purpose while the details can be found in \cite{Wal86}.

\begin{deff}{\rm
Let $(E, \mathcal E,\nu)$ be a $\sigma$-finite measure space. A white noise based on $\nu$ is a random set function $W$ on the set $A\in \mathcal E$ of finite $\nu$-measure such that
\begin{itemize}
\item W(A) is an $N(0,\nu(A))$ random variable,
\item if $A\cap B=\emptyset$, then $W(A)$ and $W(B)$ are independent and
$$
W(A\cap B)=W(A)+W(B).
$$
\end{itemize} }
\end{deff}

\begin{deff}{\rm
Let $E=\R_+^n$, $\nu$ be Lebesgue measure, and $W$ be a white noise on $E$. The Brownian sheet on $\R_+^n$ is the process $\{W_t:t\in\R_+^n\}$ defined by $W_t:=W((0,t])$, where
$t=(t_1,\dots,t_n)$,
$(0,t]:=(0,t_1]\times\dots\times(0,t_n]$.
That is a mean-zero Gaussian process. Moreover, if $s=(s_1,\dots,s_n)$ and $t=(t_1,\dots,t_n)$, the covariance function is
$$
\E(W_sW_t)=(s_1\wedge t_1)\dots(s_n\wedge t_n).
$$ }
\end{deff}

The integral in Walsh's sense is defined based on martingale measure theory, which
 is constructed as follows.

\begin{deff}{\rm
Let $U(A,\omega)$ be a (random) function defined on $\mathcal A\times \Omega$, where $\mathcal A\subset\mathcal E$ is an algebra and such that $\E (U^2(A))<\infty$, $\forall A\in \mathcal A$ and $U(A\cup B)=U(A)+U(B)$\a.s for all $A,B\in\mathcal A$, $A\cap B=\emptyset$.
We say that $U$ is $\sigma$-finite if there exists an increasing sequence $E_n\subset \mathcal E$ whose union is $E$ such that for all $n$
\begin{itemize}
\item $\mathcal E_n\subset \mathcal A$ where $\mathcal E_n:=\mathcal E_{|E_n}$,
\item $\sup \{\|U(A)\|_2: A\in\mathcal E_n\}<\infty$,
where
$
\|U(A)\|_2:=\E \left(U^2(A)\right)^{1/2}.
$
\end{itemize}
Moreover, if $U$ is countably additive on $\mathcal E_n$, $\forall n$, we can take an extension as follows. If $A\in \mathcal E$, $U(A)=\lim_{n\to\infty}U(A\cap E_n)$ if the limit does exist in $L^2$ (the space $L^2(\Omega,\F,\PP)$ endowed with the above norm) and $U(A)$ is not defined otherwise.
Such a $U$ is said to be a ``$\sigma$-finite $L^2$-valued measure".}
\end{deff}

\begin{deff} {\rm (\cite[Chapter 1]{Wal86})
Let $\F_t$ be a right continuous filtration. A process $\{M_t(A), \F_t, t\geq t, A\in \A\}$ is a martingale measure if
\begin{itemize}
\item $M_0(A)=0$,
\item if $t>0$, $M_t$ is a $\sigma$-finite $L^2$-valued measure, and
\item $\{M_t(A), \F_t, t\geq 0\}$ is a martingale.
\end{itemize} }
\end{deff}

\begin{deff}{\rm
A martingale measure $M$ is orthogonal if for any two disjoint sets $A$ and $B$, the martingales $\{M_t(A), \F_t, t\geq 1\}$ and $\{M_t(B),\F_t, t\geq 1\}$ are orthogonal.}
\end{deff}

Let $W$ be a white noise in $\R_+\times E$
and $M_t(A)=W([0,t]\times A)$. Then it is clear that $M_t(A)$ is a martingale measure. Moreover,  $M_t(A)$ and $M_t(B)$ are independent and orthogonal provided $A\cap B=\emptyset$.
It is also worthwhile to note that we can integrate over $dx$ for fixed $t$ as in
 the Bochner integral and integrate over $dt$ for fixed set $A$
as in the It\^o integral.
However, we wish to integrate over $dx$ and $dt$ together.
It is not possible to construct a stochastic integral with respect to all martingale measures.
Hence, the following class of martingale measures is defined.

\begin{deff} {\rm(\cite[Chapter 2]{Wal86})
The covariance function of $M$ is defined by
$$
\bar Q_t(A,B):=\langle M(A),M(B)\rangle_t.
$$
For a rectangle, i.e., the set $A\times B\times (s,t]\in\mathcal E\times \mathcal E\times \R_+$, define a set function $Q$ on rectangle by
$$
Q(A\times B\times (s,t]):=\bar Q_t(A,B)-\bar Q_s(A,B),
$$
and extend $Q$ by additivity to finite disjoint union of rectangles.}

\end{deff}

\begin{deff} {\rm (\cite[Chapter 2]{Wal86})
A martingale measure $M$ is ``worthy" if there exists a random $\sigma$-finite measure $K(\Lambda,\omega)$, $\Lambda\in \mathcal E\times\mathcal E\times \mathcal B$, where $\mathcal B$ consists of Borel sets on $\R_+$ such that
\begin{itemize}
\item $K$ is positive definite and symmetric in the first and the second variables,
\item for fixed $A,B\in \mathcal E$, $\{K(A\times B\times (0,t], t\geq 0)\}$ is predictable,
\item for all $n\in \N$,  $\E\left(K(E_n\times E_n\times [0,T])\right)<\infty$, where $E_n\in \mathcal E$,
\item for any rectangle $\Lambda$, $|Q(\Lambda)|\leq K(\Lambda)$.
\end{itemize}
We call $K$ the dominating measure of $M$.}
\end{deff}

Now, let $M$ be a worthy martingale measure on the Lusin space $(E,\mathcal E)$,
and
$Q_M$ and $K_M$ be its covariance and dominating measure, respectively.
The stochastic integral (in Walsh's sense) will be defined for the class of simple functions first.

\begin{deff}{\rm
A function $f$ is elementary if it is of the form
$$
f(s,x,\omega)=X(\omega)\1_{(a,b]}(s)\1_A(x),
$$
where $0\leq a\leq b$, $X$ is bounded and $\F$-measurable and $A\in\mathcal E$.
A function $f$ is simple if it is a finite sum of elementary functions.}
\end{deff}

\begin{deff}{\rm
The predictable $\sigma$-field $\mathcal P$ on $\Omega\times E\times \R_+$ is the $\sigma$-field generated by class of simple function.
A function is predictable if it is $\mathcal P$-measurable.

Let $\mathcal P_M$ be the class of all predictable functions $f$ such that $\|f\|_M<\infty$, where
$$
\|f\|_M:=\E\left((|f|,|f|)_K\right)^{1/2},
$$
and
$$
(f,g)_K:=\int_{E\times E\times \R_+}f(s,x)g(s,y)K(dxdyds).
$$ }
\end{deff}

\begin{prop}\label{prop-21} {\rm(\cite[Proposition 2.3]{Wal86})}
The class of simple function is dense in $\mathcal P_M$.
\end{prop}

For an elementary function $f(s,x,\omega)=X(\omega)\1_{(a,b]}(s)\1_A(x)$, the martingale measure $f\cdot M$ is defined by
$$
f\cdot M_t(B):=X(\omega)\left(M_{t\wedge b}(A\cap B)-M_{t\wedge a}(A\cap B)\right).
$$

\begin{prop} {\rm(\cite[Lemma 2.4]{Wal86})}
The martingale measure $f\cdot M$ is worthy and
$$
\E \left((f\cdot M_t(B))^2\right)\leq \|f\|_M^2,\;\forall B\in\mathcal E, t\leq T.
$$
\end{prop}

Now, for simple function $f$, we can define $f\cdot M$ by linearity.
Since Proposition \ref{prop-21}, we are able to define $f\cdot M$ for all $f\in\mathcal P_M$ as usual.
Finally, we define the stochastic integral by
$$
\int_0^t \int_Af(s,x)M(ds,dx):=f\cdot M_t(A),
$$
and
$$
\int_0^t\int_E f(s,x)M(ds,dx):=f\cdot M_t(E).
$$

\subsection{Equivalence of the two approaches}\label{sec-23}
We proceed with the equivalence of the stochastic integrals by Da Prato and Zabczyk  (with respect to standard cylindrical $Q$-Wiener processes) and the stochastic integrals in Walsh's sense (with respect to space-time white noises or Brownian sheets associated to the cylindrical Wiener-processes).

Now, let us assume $Q$ is the identity operator on the space $V=L^2(U)$, with
$$
U:=\left\{x=(x_1,\dots,x_d)\in\R^d: 0\leq x_i\leq 1, i=1,\dots,d\right\},
$$
and $W$ is a standard cylindrical $Q$-Wiener process and $B_W(t)$ is the associated cylindrical Wiener process.
Moreover, we define
$$
B(t,x):=\sum_{k=1}^\infty \beta_k(t)\int_{R(x)}e_k(y)dy,
$$
where
$\{e_k\}_{k=1}^\infty$ is an orthonormal basis of $L^2(U)$,
$R(x)$ is the rectangle in $U$, i.e.,
$$
R(x):=\left\{a=(a_1,\dots,a_d)\in U: 0\leq a_i\leq x_i, i=1,\dots,d\right\}.
$$
Then, it is easy to verify that $B(\cdot,\cdot)$ is a Brownian sheet; see e.g., \cite[Section 4.3.3]{PZ92}.

Consider a real-valued stochastic process $\varphi(s,x)$, $s\in [0,T]$, $x\in U$ and assume that $\varphi(s,\cdot)$, $s\in [0,T]$ is an $L^2(U)$-valued predictable process and such that
$$
\E\left(\int_0^T\int_U \varphi^2(s,x)dsdx\right)=\E\left(\int_0^T\|\varphi(s,\cdot)\|_{L^2(U)}^2ds\right)<\infty.
$$
Then, one has
\begin{equation}\label{eq-equi-integral}
\int_0^T \int_U \varphi(s,x)B(ds,dx)=\int_0^T \langle \varphi(s,\cdot),dW(s,\cdot)\rangle_{L^2(U)}=\int_0^T \varphi(s)dB_W(s),
\end{equation}
where the first integral is the stochastic integral with respect to the Brownian sheet in Walsh's sense in Section \ref{sec-22}, the second is the stochastic integral with respect to the cylindrical $Q$-Wiener process in Da Prato's and Zabczyk's sense in Section \ref{sec-21} (see Definition \ref{inf-def-1} and Proposition \ref{inf-prop-1}) and the last one is the stochastic integral with respect to the cylindrical Wiener process in the sense of Section \ref{sec-21}.
To gain more insight, the reader is referred to \cite{DQ11,PZ92}.

\para{Solutions of the two approaches and their equivalence.}
Now,
we demonstrate that the solutions of stochastic heat equation in one dimension in these approaches are equivalent. Actually, this fact holds for
large classes of SPDEs in general (including stochastic heat equation and stochastic wave equation with dimension $\leq 3$).

Consider a class of non-linear SPDEs of the following form
\begin{equation}\label{23-eq1}
\frac{\partial u(t,x)}{\partial t}=Au(t,x)+b(u(t,x))+\sigma(u(t,x))\dot W(t,x),
\end{equation}
where $t>0$, $x\in\0\subset \R^d$, $A=\Delta$ together
with some boundary condition on $\0$ if $\0$ is a bounded domain (in fact, we can consider non-linear wave equation
with the assumptions
$d\leq 3$),
$b(\cdot)$ and $\sigma(\cdot)$ are continuous,
and $\dot W(t,x)$ is a space-time white noise.

Let $T_t(x,y)$ be a fundamental solution of the problem $\frac{\partial u(t,x)}{\partial t}=Au(t,x)$ and $e^{tA}$ be the analytic semi-group generated by $A$, alternatively defined by
\begin{equation*}
(e^{tA}u)(x):=\int_\0 T_t(x,y)u(y)dy.
\end{equation*}

\para{Mild random field solution.}  A mild random field solution
$\{u(t,x): (t,x)\in [0,T]\times \0\}$ of \eqref{23-eq1} is such that the following stochastic integral equation is satisfied
$$
\begin{aligned}
u(t,x)=&\int_\0 T_t(x,y)u_0(y)dy+\int_0^t\int_\0 T_{t-s}(x,y)b(u(s,y))dyds\\
&+\int_0^t\int_\0T_{t-s}(x,y)\sigma(u(s,y))W(ds,dy).
\end{aligned}
$$
In the above, the first and the second integrals are understood as usual and
the last one is the stochastic integral in Walsh's sense (that
 is, the two parameters in integration are taken at the same time).

\para{Mild $L^2(\0)$-valued solution.} A
mild $L^2(\0)$-valued solution
$\{u(t,\cdot): t\in [0,T]\}$, $u(t,\cdot)\in L^2(\0)$ of \eqref{23-eq1} is such that the following stochastic integral equation
is satisfied (in $L^2(\0)$)
$$
\begin{aligned}
u(t)=e^{tA}u_0+\int_0^t e^{(t-s)A}b(u(s))ds+\int_0^t e^{(t-s)A}\sigma(u(s))dW(s).
\end{aligned}
$$
In the above, the second integral is a Bochner integral while the last integral is a stochastic integral in the sense of infinite dimensional integration theory in Section \ref{sec-21} (with $\sigma(u(s))$ being understood as a multiplication operator).
To end this subsection, we state the following Proposition.

\begin{prop}
Consider $\0=(0,1)$, $A=\frac{\partial^2}{\partial x^2}$ endowed with homogeneous Neumann boundary condition. The mild random field solution and the mild $L^2(\0)$-valued solution are equivalent if one of them exists uniquely
and has continuous paths (in both space and time), i.e., $u(s,\cdot)\in C\big([0,t],C([0,1],\R)\big) \a.s$ and satisfies
\begin{equation}\label{eq-cond-eqi}
\sup_{[0,T]\times[0,1]}\E (|u(t,x)|^2)<\infty.
\end{equation}
The equivalence is in the sense that if we let $u(t,x)$ be the mild random field solution then $u(t):=u(t,\cdot)$ is the mild $L^2(\0)$-valued solution and vice versa.
\end{prop}

The above Proposition follows the equivalence of stochastic integrals in random field approach and in infinite-dimensional approach (as in \eqref{eq-equi-integral}).
The condition ``has continuous paths" and \eqref{eq-cond-eqi} may be a bit restrictive.
In fact, we imposed this condition to prove the equivalence without much effort.
In certain cases, this condition may not be needed and one can verify directly that the ``mild random field solution" is equivalent to the ``mild $L^2(\0)$-valued solution".
For the details of
the proof of  this Proposition,
the reader is referred to \cite[Proposition 4.9]{DQ11}.

\subsection{Malliavin calculus}\label{sec:Amal}
We describe briefly the Malliavin calculus in this section for our own
purpose, and refer to \cite{Nua06} for a complete presentation of this subject.
Denote by $\mathcal S$ the space of smooth random variables such that for $F\in\mathcal S$, $F$ has the form
$$
F=f(W(h_1),\dots,W(h_n)),
$$
where $f\in C_b^\infty(\R^n)$,
 and $h_1,\dots,h_n$ is an orthonormal sequence in $L^2(\R_+\times(0,1))$, and $W(t,x)$ is a Brownian sheet\footnote{For sake of simplicity of notation, in this section, we can assume the Brownian sheet is the canonical process},
 and for $h\in L^2(\R_+\times(0,1))$,
$$
W(h):=\int \int h(s,y)W(ds,dy).
$$
For $F\in \mathcal S$, the first-order Mallliavin derivative $DF$ is defined to be the $L^2(\R_+\times(0,1))$-valued random variable as follows
$$
D_{t,x}F
:=\sum_{k=1}^n \partial_kf(W(h_1),\dots,W(h_n))h_k(t,x).
$$
Let $\D^{1,2}$ be the completion of $\mathcal S$ with respect to the semi-norm
$$
\|F\|_{1,2}^2:=\E|F|^2+\E|DF|_{L^2(\R_+\times (0,1))}^2.
$$
Moreover, for each $h\in L^2(\R_+\times(0,1))$, we define $D_hF$ (in fact, it is also the directional derivative) by
$$
D_hF:=\sum_{k=1}^\infty \langle h_k,DF\rangle_{L^2(\R_+\times(0,1))}\langle h_k,h\rangle_{L^2(\R_+\times(0,1))}=\langle DF,h\rangle_{L^2(\R_+\times(0,1))}.
$$
The operator $D_h$ can be extended as a closed operator with domain $\D^h$ and $\D^{1,2}\subset\D^h$.
In addition, one has
$$
D_{t,x}F=\sum_{k=1}^\infty \langle DF,h_k\rangle_{L^2(\R_+\times(0,1))}h_k(t,x)=\sum_{k=1}^\infty D_{h_k}Fh_k(t,x)
$$
(if one of them exists).
Furthermore, $F\in\D^{1,2}$ if and only if $F\in\D^{h_k}$ for each $k$ and
$$
\sum_{k=1}^\infty \E \left(|D_{h_k}F|^2\right)<\infty.
$$

Since our system has non-Lipschitz and unbounded coefficients,
we need to localize the system.
The ``local" criteria for absolute continuity of the law of a random variable is stated as follow.
\begin{deff}
{\rm (see \cite[Definition 2.1]{PT93})
A random variable $F$ is said to belong to the class $\D^{1,2}_\lo$ if there exists a sequence of measurable subsets of $\Omega$:  $\Omega_n\subset \Omega_{n+1}$ and $\cup_{n}\Omega_n=\Omega\a.s$ and a sequence $\{F_n\}\subset \D^{1,2}$ such that
$$
F|_{\Omega_n}=F_n|_{\Omega_n}\forall n.
$$
We say that $F$ is localized by the sequence $\{(\Omega_n,F_n), n\in\N\}$.}
\end{deff}

\begin{prop}\label{den-prop1}{\rm(\cite[Proposition 2.2]{PT93})}
Let $F\in \D^{1,2}_\lo$. There exists a unique measurable function of $(t,x,\omega)$ $DF$ such that for any localizing sequence $(\Omega_n,F_n)$,
$$
\1_{\Omega_n}DF=\1_{\Omega_n}DF_n.
$$
\end{prop}

\begin{prop}\label{den-prop2}
{\rm(\cite[Proposition 2.3]{PT93})}
Let $F$ be a real random variable.
A sufficient condition for the law of $F$ to be absolutely continuous with respect to the Lebesgue measure is that
\begin{itemize}
\item [\rm{(i)}]$F\in\D^{1,2}_\lo$
\item [\rm{(ii)}] $\|DF\|_{L^2(\R_+\times(0,1))}>0\a.s$
\end{itemize}
\end{prop}

To close this section, we state the following chain rule, which is used
in  Section \ref{sec:den}.

\begin{prop} {\rm(\cite[Proposition 1.2.2]{Nua06})}
Let $\varphi:\R^n\to\R$ be a continuously differentiable function with bounded partial derivative and fixed $p\geq 1$. Suppose that $F=(F_1,\dots,F_n)$ is a random vector, whose components are in $\D^{1,p}$. Then, $\varphi(F)\in\D^{1,p}$ and
$$
D\varphi(F)=\sum_{i=1}^m\frac{\partial\varphi}{\partial x_i}(F)DF_i.
$$
\end{prop}

\end{document}